\documentclass[11pt]{amsart}

\usepackage{amsfonts,graphics,amsmath,amsthm,amsfonts,amscd,amssymb,amsmath,latexsym,multicol,
 mathrsfs}
\usepackage{epsfig,url}
\usepackage{flafter}
\usepackage{fancyhdr}
\usepackage{hyperref}
\usepackage{amsmath,bm}
\usepackage{tikz}
\usepackage[section]{placeins}
\hypersetup{colorlinks=true,colorlinks=true,linkcolor=blue,urlcolor=blue, citecolor=blue}

\makeatletter

\def\jobis#1{FF\fi
  \def\predicate{#1}%
  \edef\predicate{\expandafter\strip@prefix\meaning\predicate}%
  \edef\job{\jobname}%
  \ifx\job\predicate
}

\makeatother

\if\jobis{proposal}%
\else
\fi

 \usepackage[matrix, arrow]{xy}
 \renewcommand{\mod}{\ \operatorname{mod}}

\DeclareMathOperator{\Supp}{Supp}

\DeclareMathOperator{\cent}{center}

\DeclareMathOperator{\mld}{mld}

\DeclareMathOperator{\lct}{lct}

 \numberwithin{equation}{section}
 \numberwithin{footnote}{section}

 \newtheorem{thm}{Theorem}[section]

 \newtheorem{cor}[thm]{Corollary}
 \newtheorem{lem}[thm]{Lemma}
 \newtheorem{prop}[thm]{Proposition}
 \newtheorem{conj}[thm]{Conjecture}

 \newtheorem{defn-prop}[thm]{Definition-Proposition}
 \newtheorem{defn-lem}[thm]{Definition-Lemma}
{
    \newtheoremstyle{upright}%
        {8pt plus2pt minus4pt}%
        {8pt plus2pt minus4pt}%
        {\upshape}%
        {}%
        {\bfseries\scshape}%
        {}%
        {1em}%
        {}%
\theoremstyle{upright}

 \newtheorem{defn}[thm]{Definition}
 
 \newtheorem{exa}[thm]{Example}

 \newtheorem{rem}[thm]{Remark}

}

 \newcommand{\C}{\mathbb C}
 \newcommand{\N}{\mathbb N}

 \newcommand{\Q}{\mathbb Q}
 \newcommand{\R}{\mathbb R}
 \newcommand{\Z}{\mathbb Z}
 \newcommand{\bM}{{\bf{M}}}
 \newcommand{\bN}{{\bf{N}}}
 \newcommand{\bB}{{\bf{B}}}
 \newcommand{\bD}{{\bf{D}}}

\title{Boundedness of klt complements on Fano fibrations over surfaces}
\author{Bingyi Chen}
\address{Bingyi Chen, Department of Mathematics,
Sun Yat-sen University,
Guangzhou, 510275, P. R. China.}
\email{chenby253@mail.sysu.edu.cn, chenby16@tsinghua.org.cn}
\begin{document}

\begin{abstract}
Let $(X,B)$ be an $\epsilon$-lc pair of dimension $d$ with  a closed point $x\in X$. Birkar and Shokurov conjectured that there is an effective Cartier divisor $H$ passing through $x$ such that $(X,B+tH)$ is lc near $x$, where $t$ is a positive real number depending only on $d,\epsilon$. We prove that this conjecture is equivalent to Shokurov's conjecture on boundedness of klt complements on Fano fibrations and we confirm it in dimension 2. As a corollary, we prove the boundedness of klt complements on Fano fibrations over surfaces.
\end{abstract}

\maketitle


\section{Introduction}

We work over the field of complex numbers $\C$. All varieties are assumed to be quasi-projective.

The theory of complements was introduced by Shokurov when he investigated log flips on threefolds \cite{Sh92}, which  
turns out to be a very powerful tool in birational geometry. It has played an important role in several breakthroughs in birational geometry, for example, the proof of the Borisov-Alexeev-Borisov conjecture (the boundedness of Fano varieties) \cite{Bi19,Bi21} and the openness of  K-semistability in families of log Fano pairs \cite{Xu20,BLX22}. See \cite{Sh00,Sh04,Liu18,HLS19,Sh20,Bi22,HL22,HLL22,CHL23} for more applications.


Shokurov \cite{Sh00,Sh04} proposed conjectures on the boundedness of lc (resp. klt) complements, i.e. the existence of lc (resp. klt) $n$-complements for some bounded natural number $n$. The boundedness of lc complements was proved by Birkar \cite{Bi19} for pairs on Fano type fibrations with boundary coefficients in a hyperstandard set, which was later generalized to more general boundary coefficients \cite{HLS19,Sh20}. On the other hand, the conjecture on the boundedness of klt complements is still wide open except some partial results (see Remark \ref{rem:1} below).




We now state precisely Shokurov's conjecture on the boundedness of klt complements for the case when boundary coefficients belong to a finite set of rational numbers.

\begin{conj}[Shokurov]\label{conj:shokurov}
Let $d$ be a natural number, $\epsilon$ be a positive real number and $\mathfrak R\subset [0,1]$ be a finite set of rational numbers. Then there is a natural number $n$ depending only on $d,\epsilon,\mathfrak R$ satisfying the following. Assume $(X,B)$ is a pair, $X\to Z$ is a contraction and $z\in Z$ is a closed point such that 
\begin{itemize}
\item $(X,B)$ is $\epsilon$-lc of dimension $d$,
\item the coefficients of $B$ are in $\mathfrak R$,
\item $X$ is of Fano type over $Z$, and
\item $-(K_X+B)$ is nef over $Z$.  
\end{itemize}
Then $K_X+B$ has a monotonic klt $n$-complement over $z$. That is, there is a $\Q$-divisor $\Lambda$ on $X$ such that over some neighbourhood of $z$ we have
\begin{itemize}
  \item $\Lambda\geq B$,
  \item $(X,\Lambda)$ is klt, and
  \item $n(K_X+\Lambda) \sim 0/Z$.
\end{itemize}
\end{conj}

\begin{rem}\label{rem:1}
(1) It is expected that $K_X+B$ has not just a klt $n$-complement but also an $\epsilon$-lc $n$-complement, even for more general boundary coefficients. See \cite[Conjecture 5]{Sh20} and \cite[Conjecture 1.1]{CH21} for stronger versions of Conjecture \ref{conj:shokurov}.

(2) The condition in Conjecture \ref{conj:shokurov} that 
\begin{itemize}
\item $(X,B)$ is $\epsilon$-lc and $-(K_X+B)$ is nef over $Z$,
\end{itemize}
can be replaced by another equivalent condition that
\begin{itemize}
\item $K_X+B$ has an $\epsilon$-lc $\R$-complement over $Z$, that is, there is an $\epsilon$-lc pair $(X,B^+)$ such that $B^+\geq B$ and $K_X+B^+\sim_{\R} 0/Z$. 
\end{itemize}
Indeed, if the former holds, as $X$ is of Fano type over $Z$, $-(K_X+B)$ is semiample over $Z$ and hence the latter follows; if the latter holds, we can run an MMP on $-(K_X+B)$ over $Z$ to make $-(K_X+B)$ nef over $Z$ and in this process the $\epsilon$-lc property of $(X,B)$ will be preserved because of the existence of $\epsilon$-lc $\mathbb R$-complement.

(3) When $\dim Z=0$, Conjecture \ref{conj:shokurov} is equivalent to the Borisov-Alexeev-Borisov conjecture \cite{Bi19,Bi21}.

(4) When $\dim Z=1$, Conjecture \ref{conj:shokurov}  was confirmed by Birkar \cite[Corollary 1.4]{Bi23}.

(5) In the toric case, Conjecture \ref{conj:shokurov}  was proved by Ambro \cite[Theorem 1.2]{Am22}.

(6) When $\dim X=2$, an affirmative answer of Conjecture \ref{conj:shokurov} was given in \cite[Theorem 1.6]{Bi04} and in \cite[Theorem 1.1]{CH21}.
   
(7) When $X=Z$, $\dim X=3$, $\epsilon\geq 1$ and $X$ is terminal, Conjecture \ref{conj:shokurov} was proved in \cite[Theorem 1.4]{HLL22}.

      
(8) When $X=Z$ and $(X\ni z,B)$ has a $\delta$-plt blow-up (Definition \ref{defn:plt}) for some fixed $\delta>0$, Conjecture \ref{conj:shokurov} follows from \cite[Theorem 1.6]{HLS19}. In particular, Conjecture \ref{conj:shokurov} holds when $X=Z$ and $(X\ni z,B)$ is an exceptional singularity (see Definition \ref{defn:exc}).

(9) When $(X,\Gamma)\to Z$ is a $(d,r,\epsilon)$-Fano type fibration (see \cite[Definition 1.1]{Bi22}) and $B\leq \Gamma$ with $-(K_X+B)$ being big over $Z$, Conjecture \ref{conj:shokurov} was proved in \cite[Theorem 1.7]{Bi22}. Roughly speaking,  Conjecture \ref{conj:shokurov} holds when $Z$ belongs to a bounded family and the  ``degree" of $K_X+B$ is in some sense bounded.

(10) When $(X,B)\to Z$ is a semi-stable morphism (\cite[2.7]{FM18}) and $mK_Z$ is Cartier for some fixed $m$, Conjecture \ref{conj:shokurov} was proved in \cite[Theorem 1.4]{FM18}.

\end{rem}

To attack the above conjecture, Birkar and Shokurov proposed the following two conjectures. Birkar pointed out to me that these two conjectures are equivalent and can imply Conjecture \ref{conj:shokurov}.


\begin{conj}[Birkar-Shokurov]\label{conj:birkar1}
Let $d$ be a natural number and $\epsilon$ be a positive real number. Then there is a positive real number $t$ depending only on $d,\epsilon$ satisfying the following. Assume $(X,B)$ is a pair, $f:X\to Z$ is a contraction and $z\in Z$ is a closed point such that 
  \begin{itemize}
  \item $(X,B)$ is $\epsilon$-lc of dimension $d$ and $\dim Z>0$,
  \item $X$ is of Fano type over $Z$, and
  \item $-(K_X+B)$ is nef over $Z$.  
  \end{itemize}
Then there exists an effective Cartier divisor $H$ on some neighbourhood $U$ of $z\in Z$ such that $z\in \Supp H$ and $(X,B+tf^*H)$ is lc over $U$.
\end{conj}

\begin{conj}[Birkar-Shokurov]\label{conj:birkar2}
Let $s$ be a natural number and $\epsilon$ be a positive real number. Then there is a positive real number $t$ depending only on $s,\epsilon$ satisfying the following. Assume $(Z,\Delta)$ is an $\epsilon$-lc pair of dimension $s$ with a closed point $z\in Z$. Then there exists an effective Cartier divisor $H$ on some neighbourhood of $z\in Z$ such that $z\in \Supp H$ and $(Z,\Delta+tH)$ is lc near $z$.
\end{conj} 

\begin{rem}\label{rem:new}
(1) Conjecture \ref{conj:birkar2} is a special case of Conjecture \ref{conj:birkar1} when $X=Z$.

(2) As in Remark \ref{rem:1}(2), in Conjecture \ref{conj:birkar1} the condition that $(X,B)$ is $\epsilon$-lc and $-(K_X+B)$ is nef over $Z$ can be replaced by another equivalent condition that $K_X+B$ has an $\epsilon$-lc $\R$-complement over $Z$.

(3) At a meeting of the COW seminar at City, University of London on 7th February 2018, Birkar raised Conjecture \ref{conj:birkar1} in the special case when $Z=\mathbb A^n$, $X\to Z$ is a weighted blow up at the origin and $B=0$ (see \cite[Conjecture 1.2]{SS21}). It was proved by G. Sankaran and F. Santos in \cite[Theorem 1.3]{SS21} that Conjecture \ref{conj:birkar1} holds in this case.

(4) Ambro had independently considered a version of these conjectures in private work. He \cite[Theorem 3.12]{Am22} proved Conjecture \ref{conj:birkar1} in the toric setting with an additional condition that the coefficients of $B$ belong to a fixed DCC set.

(5) When $\dim Z=1$, Conjecture \ref{conj:birkar1} was proved by Birkar in \cite[Theorem 1.1]{Bi23}. See \cite[Theorem 1.4]{Bi16} for another proof of Conjecture \ref{conj:birkar1} in the special case when $\dim Z=1$, $K_X+B\sim_{\R} 0/Z$ and $(F,\Supp B|_F)$ is log bounded where $F$ is a general fiber of $f$.
\end{rem}


\noindent\textbf{Relations among the above conjectures.} Our first main result is on the relations among Conjecture \ref{conj:shokurov}, \ref{conj:birkar1} and \ref{conj:birkar2}.

\begin{thm}\label{thm:main}
Let $\ell$ be a natural number. The following are equivalent:
\begin{enumerate}
\item Conjecture \ref{conj:shokurov} holds when $\dim Z \leq \ell$;
\item Conjecture \ref{conj:shokurov} holds when $X\to Z$ is birational, $B=0$ and $\dim Z \leq  \ell$.
\item Conjecture \ref{conj:birkar1} holds when $\dim Z \leq  \ell$.
\item Conjecture \ref{conj:birkar1} holds when $X\to Z$ is birational, $B=0$ and $\dim Z \leq  \ell$.
\item Conjecture \ref{conj:birkar2} holds when $\dim Z \leq  \ell$.
\end{enumerate}
\end{thm}

It is worth to mention that the proof of the above theorem relies on \cite[Theorem 1.1]{Bi23}, which relates the singularities on the base space of a Fano type fibration with the singularities on the total space.

\medskip

\noindent\textbf{The surface case.} We confirm Conjecture \ref{conj:birkar2} in dimension 2. As a corollary, we prove the boundedness of klt complements on Fano fibrations over surfaces.

\begin{thm}\label{thm:surface}
Let $(Z,\Delta)$ be a pair of dimension 2 with a closed point $z\in Z$ such that $\mld(Z\ni z, \Delta)\geq \epsilon$  with $0<\epsilon\leq 1$. Then there exists an effective Cartier divisor $H$ on some neighbourhood of $z\in Z$ such that $z\in \Supp H$ and $(Z,\Delta+tH)$ is lc near $z$, where $t=\epsilon^2/24$.
\end{thm}

Here we denote by $\mld(Z\ni z, \Delta)$ the infimum of the log discrepancy of $E$ with respect to $(Z,\Delta)$ where $E$ runs over all prime divisors over $Z$ with $\cent_Z E=z$. Note that the condition $\mld(Z\ni z, \Delta)\geq \epsilon$ is weaker than that $(Z,\Delta)$ is $\epsilon$-lc near $z$.

\begin{cor}\label{cor:dim2}
Conjecture \ref{conj:shokurov}, \ref{conj:birkar1} and \ref{conj:birkar2} hold when $\dim Z\leq 2$.
\end{cor}

\begin{rem}
The following example indicates that the order $O(\epsilon^2)$ in Theorem \ref{thm:surface} is optimal.
\end{rem}

\begin{exa}\label{exa}
Let $Z=\mathbb A^2$ with coordinates $x,y$ and let $z\in Z$ be the origin. Let 
$$\Delta=\frac{2m-1}{m^2}\cdot (x^m+y^{m+1}=0)$$
where $m\in \mathbb N$.
By \cite[Lemma A.1]{Ch22} we have 
$$\mld(Z\ni z,\Delta)=\inf\{p_1+p_2-\min\{\lambda m p_1,\lambda (m+1) p_2\} \mid (p_1,p_2)\in \mathbb N^2 \}$$
where $\lambda=\frac{2m-1}{m^2}$. It is not hard to check that the above expression is equal to $\frac{1}{m}$. So we have 
$$\mld(Z\ni z,\Delta)=\frac{1}{m}.$$
Let $E$ be the exceptional divisor obtained by the weighted blow up at $z$ with $\textbf{wt}(x,y)=(m+1,m)$. Then
$a(E,Z,\Delta)=\frac{1}{m}$. Moreover, the vanishing order of the pullback of $x$ (resp. $y$) along $E$ is $m+1$ (resp. $m$). So for any curve $H$ on $Z$ passing through $z$, $(Z,\Delta+tH)$ is not lc near $z$ when $t>\frac{1}{m^2}$.
\end{exa}

\medskip

\noindent\textbf{The toric case.} As mentioned in Remark \ref{rem:new}, in the toric case Conjecture \ref{conj:birkar1} was proved by Ambro \cite[Theorem 3.12]{Am22} with an additional condition that the coefficients of $B$ belong to a fixed DCC set. In the following theorem we remove this additional condition.


\begin{thm}\label{thm:toric} Conjecture \ref{conj:birkar1} holds when $f:X\to Z$ is a toric contraction between toric varieties, $(X,B)$ is a toric pair and $z\in Z$ is a torus invariant closed point.
\end{thm}

\begin{cor}\label{cor:toric} Conjecture \ref{conj:birkar2} holds when $(Z,\Delta)$ is a toric pair and $z\in Z$ is a torus invariant closed point.
\end{cor}

\begin{rem}

(1) See Theorem \ref{thm:toric2} for a stronger version of Theorem \ref{thm:toric}, which asserts that Conjecture \ref{conj:birkar1} holds in the toric case, even when we replace the condition that $(X,B)$ is $\epsilon$-lc by a weaker condition that $\mld(X/Z\ni z,B)\geq \epsilon$, that is, $a(E,X,B)\geq \epsilon$ for any prime divisor $E$ over $X$ with $f(\cent_X E)=z$.

(2) After this paper was completed, Ambro informed me that he 
had proved Theorem \ref{thm:toric}  with an effective estimate for $t$ (under the weaker condition that $\mld(X/Z\ni z,B)\geq \epsilon$). 
\end{rem}

\medskip

\noindent\textbf{The case for exceptional singularities.} 
We confirm Conjecture \ref{conj:birkar2} for singularities admitting $\delta$-plt blow-ups (Definition \ref{defn:plt}) for some fixed $\delta>0$. As a corollary, we prove the boundedness of klt complements over an exceptional singularity (Definition \ref{defn:exc}) $Z\ni z$ of dimension 3.

\begin{thm}\label{thm:exceptional}
Fix a positive real number $\delta$. Conjecture \ref{conj:birkar2} holds when $(Z\ni z,\Delta)$ has a $\delta$-plt blow-up.
\end{thm}

\begin{thm}\label{thm:exc2}
Fix a DCC set $\mathfrak D\subset [0,1]$. Conjecture \ref{conj:birkar2} holds when there is a boundary $\Omega$ on $Z$ such that 
\begin{itemize}
\item $\Omega\leq \Delta$,
\item $(Z\ni z,\Omega)$ is an exceptional singularity, and 
\item the coefficients of $\Omega$ are in $\mathfrak D$.
\end{itemize} 
\end{thm}

\begin{cor}\label{cor:exc1}
Conjecture \ref{conj:birkar1} and \ref{conj:birkar2} hold when $Z\ni z$ is an exceptional singularity.
\end{cor}

\begin{cor}\label{cor:exc2}
  Conjecture \ref{conj:shokurov} holds when $Z\ni z$ is an exceptional singularity and $\dim Z=3$.
\end{cor}

\subsection*{Acknowledgements} I would like to express my sincere gratitude to Caucher Birkar for suggesting this question, and for lots of helpful discussions and insightful suggestions. 
I am also grateful to Junpeng Jiao, Xiaowei Jiang, Minzhe Zhu and  Florin Ambro for many useful discussions. This work was mainly done at Tsinghua University and was revised at Sun Yat-sen University. I am supported by the start-up fund from Sun Yat-sen University.

\section{Preliminaries}
We will freely use the standard notations and definitions in \cite{KM98,BCHM}. A contraction $f:X\rightarrow Z$ is a projective morphism of normal varieties with $f_*\mathcal{O}_X=\mathcal{O}_Z$. 

\subsection{Divisors and ideals}
Let $X$ be a normal variety, $D$ be an $\R$-divisor on $X$ and $\mathfrak a$ be a non-zero ideal sheaf on $X$. For a prime divisor $T$ on $X$, we denote by $\mu_T D$ the coefficient of $T$ in $D$ and denote by $\mu_T \mathfrak a$ the smallest vanishing order of $s$ along $T$, where $s$ runs over all sections of $\mathfrak a$ at the generic point of $T$. For a prime divisor $T$ over $X$, i.e. a prime divisor on some birational model $\pi:Y\to X$, by $\mu_T \mathfrak a$ we mean  $\mu_T \pi^*\mathfrak a$ and by $\mu_T D$ we mean $\mu_T \pi^*D$ in the case when $D$ is an $\mathbb R$-Cartier $\mathbb R$-divisor on $X$.

\subsection{Pairs and singularities} 

\begin{defn}
A sub-pair $(X,B)$ consists of a normal variety $X$ and an $\R$-divisor $B$ on $X$ such that $K_X+B$ is $\R$-Cartier. A sub-pair $(X,B)$ is called a pair if $B$ is effective. If $(X,B)$ is a pair and the coefficients of $B$ are at most $1$, we call $B$ a boundary on $X$.

\end{defn}

\begin{defn} 
Let $(X,B)$ be a sub-pair and $E$ be a prime divisor on some birational model $\pi:Y\to X$. Denote by $\cent_X E$ the center of $E$ on $X$, i.e. the image of $E$ on $X$ under the morphism $\pi$.
We may write $K_Y+B_Y=\pi^*(K_X+B)$ for some uniquely determined $B_Y$. Then the log discrepancy $a(E,X,B)$ of $E$ with respect to $(X,B)$ is defined as $1-\mu_E B_Y$.
\end{defn}

\begin{defn}
Let $\epsilon$ be a non-negative real number. 
A sub-pair $(X,B)$ is said to be $\epsilon$-lc (resp. $\epsilon$-klt, lc, klt) if $a(E,X,B)\geq \epsilon$ (resp. $>\epsilon$,~$\geq 0$,~$> 0$) for any prime divisor $E$ over $X$. In the case when $B = 0$, we also say $X$ is $\epsilon$-lc (resp. $\epsilon$-klt, lc, klt). A pair $(X,B)$ is said to be plt (reps. $\epsilon$-plt) if $a(E,X,B)>0$ (resp. $>\epsilon$) for any exceptional prime divisor $E$ over $X$.
\end{defn}

\begin{defn}
A germ $(X\ni x,B)$ consists of a pair $(X,B)$ and a closed point $x\in X$. We say $(X\ni x,B)$ is an $\epsilon$-lc (resp. $\epsilon$-klt, lc, klt) germ if $(X,B)$ is $\epsilon$-lc (resp. $\epsilon$-klt, lc, klt) near $x$. When $B=0$, we use $X\ni x$ instead of $(X\ni x,0)$.
\end{defn}

\begin{defn}\label{defn:try}
Let $(X,B)$ be a sub-pair, $f:X\to Z$ be a projective morphism and $z\in Z$ be a (not necessarily closed) point. We say $(X,B)$ is $\epsilon$-lc (or $\epsilon$-klt, lc, klt) over $z$ if $(X,B)$ is so over some neighbourhood of $z$. We call $E$ a prime divisor over $X/Z\ni z$ if $E$ is a prime divisor over $X$ with $f(\cent_X E)=\overline{z}$ and we denote
\begin{align*}
  \mld(X/Z\ni z,B):=\inf\{a(E,X,B) \mid ~ &\text{$E$ is a prime divisor over $X/Z\ni z$}\}.
\end{align*}
In the case when $Z=X$, $z=x$ and $f$ is the identity morphism, we call $E$ a prime divisor over $X\ni x$ if $E$ is a prime divisor over $X$ with $\cent_X E=\overline{x}$ and we use $\mld(X\ni x,B)$ instead of $\mld(X/Z\ni z,B)$. 
\end{defn}

We remark that $\mld(X/Z\ni z,B)\geq 0$ if and only if $(X,B)$ is lc over $z$ (see \cite[Lemma 2.8]{HJL22}). However, $\mld(X/Z\ni z,B)\geq \epsilon$ is weaker than $(X,B)$ is $\epsilon$-lc over $z$.

\subsection{b-divisors}
Let $X$ be a normal variety. A b-divisor $\textbf{D}$ over $X$ is a collection of $\R$-divisors $\textbf{D}_Y$ on $Y$, where $Y$ runs over all birational models over $X$, such that $\sigma_*\textbf{D}_{Y_1}=\textbf{D}_{Y_2}$ for any birational morphism $\sigma:Y_1\rightarrow Y_2/X$.  

Let $\bD$ be a b-divisor over $X$ and $Y_0$ be a birational model over $X$. We say $\bD$ descends on $Y_0$ if $\textbf{D}_{Y_0}$ is an $\R$-Cartier $\R$-divisor and $\textbf{D}_{Y}=\sigma^* \textbf{D}_{Y_0}$ for any birational morphism $\sigma:Y\rightarrow Y_0/X$.

Let $X\rightarrow U$ be a projective morphism. We say that a b-divisor $\textbf{D}$ over $X$ is b-nef$/U$ if $\bD$ descends on some birational model $Y_0$ over $X$ and $\textbf{D}_{Y_0}$ is nef$/U$. 

We denote by $\textbf{0}$ the b-divisor $\bD$ such that $\textbf{D}_Y=0$ for every birational model $Y$ over $X$.

\subsection{Generalized pairs}

\begin{defn} 
A generalized sub-pair (g-sub-pair for short) $(X,B+\bM)/U$ consists of a normal variety $X$ associated with a projective morphism $X\rightarrow U$, an $\R$-divisor $B$ on $X$, and a  b-nef$/U$ b-divisor $\bM$ over $X$.
A g-sub-pair $(X,B+\bM)/U$ is called a generalized pair (g-pair for short) if $B$ is effective.

We may drop $U$ when we emphasize the structures of $(X,B+\bM)$ that are independent of the choice of $U$, for example, the singularities of $(X,B+\bM)$.
\end{defn}

\begin{defn} \label{sing}
Let $(X,B+\bM)$ be a g-sub-pair and $E$ be a prime divisor on some birational model $\pi:Y\to X$. We can write 
$$K_Y+B_Y+\bM_Y=\pi^*(K_X+B+\bM_X)$$
for some uniquely determined $B_Y$. Then the log discrepancy of $E$ with respect to $(X,B+\bM)$ is defined as $1-\mu_E B_Y$ and denoted by $a(E,X,B+\bM)$.
\end{defn}

\begin{defn}
Let $\epsilon$ be a non-negative real number. 
A g-sub-pair $(X,B+\bM)$ is said to be $\epsilon$-lc (resp. $\epsilon$-klt, lc, klt) if $a(E,X,B+\bM)\geq \epsilon$ (resp. $>\epsilon$,~$\geq 0$,~$> 0$) for any prime divisor $E$ over $X$. 
\end{defn}

In the case that $\bM=\textbf{0}$, all definitions of g-pairs coincide with those of usual pairs.

\subsection{Canonical bundle formula}
Let $f:X\rightarrow Z$ be a contraction between normal varieties  with $\dim Z>0$ and let $Z\to U$ be a projective morphism. Let $(X,B)$ be a pair such that $K_X+B\sim_{\R} 0/Z$ and $(X,B)$ is lc over the generic point of $Z$. Then  $K_X+B\sim_{\R} f^* L$ for some $\R$-Cartier $\R$-divisor $L$ on $Z$. 
  
For any prime divisor $D$ on $Z$, denote by $t_D$ the largest number such that $(X,B+t_Df^*D)$ is lc over the generic point of $D$. This makes sense even if $D$ is not $\Q$-Cartier because we only need the pullback of $D$ over the generic point of $D$ where $Z$ is smooth. We set $B_Z=\sum_D (1-t_D)D$ where $D$ runs over all prime divisors on $Z$
and set $M_Z=L-K_Z-B_Z$. The former is called the discriminant divisor and the latter is called the moduli divisor. By construction, we have the so called canonical bundle formula
$$K_X+B\sim_{\R} f^*(K_Z+B_Z+M_Z).$$
  
Let $\sigma: Z'\rightarrow Z$ be a birational morphism from a normal variety $Z'$ and let $X'$ be a resolution of the main component of $X\times_Z Z'$ with induced morphisms $\tau:X'\rightarrow X$ and $f':X'\rightarrow Z'$. Write $K_{X'}+B'=\tau^*(K_X+B)$, then $K_{X'}+B'\sim_{\R} f'^* \sigma^* L$. Similarly we can define the discriminant divisor $B_{Z'}$ and the moduli divisor $M_{Z'}$ for the contraction $(X',B')\rightarrow Z'$. One can check that $\sigma_* B_{Z'}=B_Z$ and
$\sigma_* M_{Z'}=M_Z$. Hence there exist b-divisors $\bB,\bM$ over $Z$ such that $\bB_{Z'}=B_{Z'}$ and $\bM_{Z'}=M_{Z'}$ for any birational model $Z'$ over $Z$, which are called the discriminant b-divisor and the moduli b-divisor respectively. 

It was shown by Kawamata \cite{Ka97,Ka98} and Ambro \cite{Am04,Am05} that $\bM$ is a b-nef$/U$ b-divisor over $Z$ (see \cite[Section 3]{Hu21} for $\R$-coefficients). Thus we can regard $(Z,B_Z+\bM)/U$ as a g-pair. We call $(Z,B_Z+\bM)/U$ the g-pair induced by the canonical bundle formula for $(X,B)\rightarrow Z/U$.


\begin{lem}\label{lem:lc1}
  Keep the notation and assumptions as above. If $(Z,B_Z+\bM)$ is an lc g-pair, then $(X,B)$ is lc.
\end{lem}
\begin{proof}
Suppose that $(X,B)$ is not lc. Since $(X,B)$ is lc over the generic point of $Z$ by assumption, there is a prime divisor $S$ over $X$ such that $a(S,X,B)<0$ and $S$ is vertical over $Z$. By \cite[Lemma 3.7]{Bi19}, there is a prime divisor $T$ over $Z$ such that $a(T,Z,B_Z+\bM)<0$. This contradicts the hypothesis that $(Z,B_Z+\bM)$ is an lc g-pair.
\end{proof}

\begin{lem}\label{lem:lc2}
Keep the notation and assumptions as above. Let $t$ be a non-negative number and $H$ be an effective $\R$-Cartier divisor on $Z$. If $(Z,B_Z+tH+\bM)$ is an lc g-pair, then $(X,B+tf^*H)$ is lc.
\end{lem}
\begin{proof}
By \cite[Lemma 7.4(iii)]{PS09}, $(Z,B_Z+tH+\bM)$ is the g-pair induced by the canonical bundle formula for $(X,B+tf^*H)\rightarrow Z$. Thus the lemma follows from Lemma \ref{lem:lc1}.
\end{proof}

\subsection{Fano type fibrations} Let $X\to Z$ be a contraction between normal varieties. We say $X$ is of Fano type over $Z$ if one of the following three equivalent conditions holds (\cite[Lemma-Definition 2.6]{PS09}):
\begin{itemize}
  \item there is a klt pair $(X,C)$ such that $-(K_X+C)$ is ample over $Z$,
  \item there is a klt pair $(X,C)$ such that $-(K_X+C)$ is big and nef over $Z$, and
  \item $-K_X$ is big over $Z$ and there is a klt pair $(X,C)$ such that $K_X+C\sim_{\R} 0/Z$.
\end{itemize}

Assume $X$ is of Fano type over $Z$ and $D$ is an $\R$-Cartier $\R$-divisor on $X$. By the base point free theorem, if $D$ is nef over $Z$, then $D$ is semi-ample over $Z$. In particular, if $D\equiv 0/Z$, then $D\sim_{\R} 0/Z$. By \cite{BCHM}, one can run an MMP on $D$ over $Z$ and the MMP ends with a good minimal model or a Mori fiber space over $Z$.

\begin{lem}\cite[Lemma 2.8]{PS09} \label{lem:preserve}
Suppose that $X$ is of Fano type over $Z$.

{\rm{(1)}} Suppose $X \dashrightarrow X'/Z$ is a birational map whose inverse does not contract any divisor. Then $X'$ is of Fano type over $Z$.

{\rm{(2)}} Suppose $X\to Y/Z$ is a contraction of varieties over $Z$. Then $Y$ is of Fano type over $Z$.

{\rm{(3)}} Let $(X,B)$ be an lc pair such that $-(K_X+B)$ is nef over $Z$. Let $Y$ be a normal variety with a projective birational morphism $Y\to X$ and let $K_Y+B_Y$ be the pullback of $K_X+B$. If every exceptional$/X$ component of $B_Y$ has positive coefficient, then $Y$ is of Fano type over $Z$.
\end{lem}

\subsection{Complements}

\begin{defn}
Let $X\to Z$ be a contraction and $(X,B)$ be a pair. A monotonic lc (resp. klt) $n$-complement of $K_X+B$ over a point $z\in Z$ is of the form $K_X+\Lambda$ such that over some neighbourhood of $z$ we have
\begin{itemize}
  \item $\Lambda\geq B$,
  \item $(X,\Lambda)$ is lc (resp. klt), and
  \item $n(K_X+\Lambda) \sim 0/Z$.
\end{itemize}
\end{defn}

\subsection{Triples}

\begin{defn}
A triple $(X,B,\mathfrak a^t)$ consists of a pair $(X,B)$ and a non-zero ideal sheaf $\mathfrak a$ on $X$ with a non-negative exponent $t$. 
\end{defn}
  
\begin{defn} 
  Let $(X,B,\mathfrak a^t)$ be a triple and $E$ be a prime divisor on some birational model $\pi:Y\to X$. We may write $K_Y+B_Y=\pi^*(K_X+B)$ for some uniquely determined $B_Y$. The log discrepancy $a(E,X,B,\mathfrak a^t)$ of $E$ with respect to $(X,B,\mathfrak a^t)$ is defined by
  $$a(E,X,B,\mathfrak a^t):=1-\mu_E B_Y-t \cdot \mu_E \pi^*\mathfrak a.$$ 
\end{defn}
  
\begin{defn}
Let $\epsilon$ be a non-negative real number. 
A triple $(X,B,\mathfrak a^t)$ is said to be $\epsilon$-lc (resp. $\epsilon$-klt, lc, klt) if $a(E,X,B,\mathfrak a^t)\geq \epsilon$ (resp. $>\epsilon$,~$\geq 0$,~$> 0$) for any prime divisor $E$ over $X$. 
\end{defn}
  
\begin{defn}
Let $(X,B,\mathfrak a^t)$ be a triple, $f:X\rightarrow Z$ be a projective morphism and $z\in Z$ be a (not necessary closed) point. We say $(X,B,\mathfrak a^t)$ is $\epsilon$-lc (or $\epsilon$-klt, lc, klt) over $z$ if $(X,B,\mathfrak a^t)$ is so over some neighbourhood of $z$. We denote
\begin{align*}
    \mld(X/Z\ni z,B,\mathfrak a^t):=\inf\{a(E,X,B,\mathfrak a^t) \mid ~ &\text{$E$ is a prime divisor over $X/Z \ni z$}\}.
\end{align*}
(Recall that a prime divisor $E$ over $X/Z\ni z$ means a prime divisor $E$ over $X$ with $f(\cent_X E)=\overline{z}$, see Definition \ref{defn:try}.) In the case that $Z=X$, $z=x$ and $f$ is the identity morphism, we will use  $\mld(X\ni x,B,\mathfrak a^t)$ instead of $\mld(X/Z\ni z,B,\mathfrak a^t)$.
\end{defn}

In the case that $\mathfrak a=\mathcal O_X$, all definitions of triples coincide with those of usual pairs.

\begin{lem}\label{lem:linearsystem}
Let $f:X\to Z$ be a contraction with $\dim Z>0$, $(X,B)$ be a pair, $z\in Z$ be a closed point and $t$ be a non-negative number. Denote by $\mathfrak m_{z}$ the ideal sheaf of the closed point $z\in Z$. Let $V_z$ be a linear system on $Z$ such that 
\begin{itemize}
  \item all the elements of $V_z$ contain $z$,
  \item the local defining equations of elements of $V_z$ generate $\mathfrak m_{z}$ near $z$, and
  \item $V_z$ is base point free outside $z$.
\end{itemize}
Then the following holds.
\begin{enumerate}
\item Let $T$ be a prime divisor over $X/Z\ni z$. Then 
$$a(T,X,B,f^*\mathfrak m_{z}^t)=a(T,X,B+tf^*H)$$
for a general element $H\in V_z$.
\item If $t\leq 1$, then
$$\mld(X/Z\ni z,B,f^*\mathfrak m_{z}^t)=\mld(X/Z\ni z,B+tf^*H)$$
for a general element $H\in V_z$. 
\item Let $\epsilon$ be a real number in $[0,1]$. If $(X,B,f^*\mathfrak m_z^t)$ is $\epsilon$-lc and $t\leq 1-\epsilon$, then $(X,B+tf^*H)$ is $\epsilon$-lc for a general element $H\in V_z$.
\end{enumerate}
\end{lem}
\begin{rem}\label{rem}
A linear system satisfying the conditions in Lemma \ref{lem:linearsystem} always exists. Indeed, since $Z$ is assumed to be quasi-projective, we can find a compactification $\bar{Z}$ of $Z$ which is projective. Let $A$ be a very ample divisor on $\bar{Z}$ and let $V_z$ be the sub-linear system of $A|_Z$ consisting of elements containing $z$. Then $V_z$ satisfies the conditions.
\end{rem}
\begin{proof}
(1) As the defining equations of elements in $V_z$ generate $\mathfrak m_{z}$ near $z$, there is $H\in V_z$ such that $\mu_T f^*H=\mu_T f^*\mathfrak m_z$ and thus the first assertion follows.

(2) Let $\pi:W\to X$ be a resolution of $X$ such that $\pi^*f^*\mathfrak m_z$ is a locally principle ideal sheaf on $W$, that is, there exists an effective divisor $F$ on $W$ such that $$\pi^*f^*\mathfrak m_{z}=\mathcal O_W(-F).$$
For any element $H\in V_z$, we have $\pi^*f^* H\geq F$.
Write $K_W+B_W=\pi^*(K_X+B)$. Then 
$$\mld(W/Z\ni z, B_W+tF)=\mld(X/Z\ni z, B, \mathfrak m_z^t).$$

Let $L_z$ be the linear system on $W$ defined by
$$L_z:=\{\pi^*f^*H-F\mid H\in V_z \}.$$
We claim that $L_z$ is base point free. Indeed, for any closed point $p\in \pi^{-1}f^{-1} z=\Supp F$, there is an element $P\in V_z$ such that the local defining equation of $\pi^*f^*P$ generates $\mathcal O_W(-F)$
in a neighbourhood of $p$, because the local defining equations of  elements in $V_z$ generate $\mathfrak m_{z}$ near $z$ and because $\pi^*f^*\mathfrak m_{z}=\mathcal O_W(-F)$. Thus $\pi^*f^*P=F$ near $p$, which implies that $\pi^*f^*P-F$ does not contain $p$. So $L_z$ is base point free on a neighbourhood of $p$ for any closed point $p\in \pi^{-1}f^{-1} z$, which implies that $L_z$ is base point free over a neighbourhood of $z$. As $V_z$ is base point free outside $z$, $L_z$ is base point free everywhere.

Let $H$ be a general element of $V_z$. Then $R:=\pi^*f^*H-F$ is a general element of $L_z$. We have
\begin{align*} 
&\mld(X/Z \ni z, B+tf^*H)=\mld(W/Z\ni z, B_W+t\pi^*f^*H)\\
=&\mld(W/Z\ni z, B_W+tF+tR)=\mld(W/Z\ni z, B_W+tF)\\
=&\mld(X/Z\ni z, B, \mathfrak m_z^t).
\end{align*}
The second equality in the second low holds because $t\leq 1$ and $L_z$ is base point free.

(3) Let $H$ be a general element of $V_z$. Since $(X,B,f^*\mathfrak m_z^t)$ is $\epsilon$-lc, $(X,B)$ is $\epsilon$-lc. It follows that $(X,B+tf^*H)$ is $\epsilon$-lc outside $f^{-1}z$  as $V_z$ is base point free outside $z$ and as $t\leq 1-\epsilon$. By (2) we have 
$$\mld(X/Z\ni z,B+tf^*H)=\mld(X/Z\ni z,B,f^*\mathfrak m_{z}^t)\geq \epsilon.$$
Therefore, for any prime divisor $T$ over $X$, no matter its image on $Z$ is $z$ or not, we have $a(T,X,B+tf^*H)\geq \epsilon$. So $(X,B+tf^*H)$ is $\epsilon$-lc.
\end{proof}

\begin{cor}\label{cor:idealdiv}
Let $f:X\to Z$ be a contraction with $\dim Z>0$, $(X,B)$ be a pair, $z\in Z$ be a closed point and $t$ be a non-negative number. Denote by $\mathfrak m_{z}$ the ideal sheaf of $z\in Z$. The following are equivalent.
\begin{enumerate}
\item There is an effective Cartier divisor $H$ on some 
neighbourhood of $z\in Z$  such that $z\in \Supp H$ and $(X,B+tf^*H)$ is lc over $z$.
\item There is an effective Cartier divisor $H$ on $Z$ such that $z\in \Supp H$ and $(X,B+tf^*H)$ is lc over $z$.
\item $(X,B,f^*\mathfrak m_{z}^t)$ is lc over $z$.
\end{enumerate}
\end{cor}

\begin{proof}
It is obvious that (2) implies (1) and (1) implies (3). So it suffices to show that (3) implies (2). Assume that $(X,B,\mathfrak m_{z}^t)$ is lc over $z$. By Lemma \ref{lem:linearsystem}(2) and Remark \ref{rem}, there is an effective Cartier divisor $H$ on $Z$ such that $z\in \Supp H$ and 
$$\mld(X/Z\ni z,B+tf^*H)=\mld(X/Z\ni z,B,\mathfrak m_{z}^t)\geq 0.$$
Thus $(X,B+tf^*H)$ is lc over $z$.
\end{proof}

\subsection{Rational surface singularities}
\begin{defn}
Let $Z$ be a normal variety with a closed point $z$. We say $Z\ni z$ is a rational singularity if for any resolution $\pi:Y\to Z$ the following holds near $z$:
$$R^i\pi_* \mathcal O_Y=0, \quad \text{for any } i>0.$$
Moreover, if $\dim Z=2$, we say $Z\ni z$ is a rational surface singularity.
\end{defn}

\begin{lem}\label{lem:good}\cite[Lemma 1.3]{Br68}
Let $Z\ni z$ be a rational surface singularity and let $Y\to Z$ be a resolution. Denote by $E$ the exceptional divisor on $Y$ over $Z\ni z$. Then $E$ is a simple normal crossing divisor and each component of $E$ is a smooth rational curve.
\end{lem}

\begin{defn-lem}\cite[pp. 131-132]{Ar66}
Let $Z$ be a normal surface with a closed point $z$. Let $\pi:Y\to Z$ be a resolution with the exceptional divisor $E=\cup_{i=1}^r E_i$ over $Z\ni z$. Denote by $\mathcal S$ the set
$$\{C=\sum_{i=1}^r a_i E_i\mid a_i\in \Z^{\geq 0}, ~C\neq 0, ~C\cdot E_k\leq 0 \text{ for } k=1,\cdots r\}.$$
Then $\mathcal S$ has a minimum element, that is, an element $C_f\in \mathcal S$ such that $C\geq C_f$ for any $C\in \mathcal S$.
We call this minimum element the fundamental cycle of $Z\ni z$ on $Y$.
\end{defn-lem}

\begin{rem}
It is obvious that the fundamental cycle is determined by the intersection matrix $(E_i \cdot E_j)^r_{i,j=1}$.
\end{rem}

\begin{rem}\label{rem:laufer}
Laufer \cite[Proposition 4.1]{La72} developed an algorithm to compute the fundamental cycle $C_f$ as follows.

(1) Begin with any of the components of the exceptional divisor  $C_1:=E_{i_0}.$\\
Assume that $C_l$ is already defined for some $l>0$. There are two cases.
  
(2) If $C_l\cdot E_k\leq 0$ for any $k$, then we stop and $C_l$ is the fundamental cycle.
  
(3) If there is a component $E_{i_l}$ such that $C_l\cdot E_{i_l}>0$, then set $C_{l+1}:=C_l+E_{i_l}$ and repeat the above process.
\end{rem}

\begin{lem}\label{lem:fundamental}\cite[Theorem 4]{Ar66}
    Let $Z\ni z$ be a rational surface singularity. Let $\pi:Y\to Z$ be a resolution and denote by $\mathfrak m_z$  the ideal sheaf of the closed point $z\in Z$. Then 
    $$\pi^*\mathfrak m_{z}=\mathcal O_Y(-C_f),$$
where $C_f$ is the fundamental cycle of $Z\ni z$ on $Y$.
\end{lem}

\subsection{Toric varieties}

A toric variety is a normal variety $X$ of dimension $d$ that contains a torus $T_X$ (that is, isomorphic to $(\C^*)^d$) as a dense open subset, together with an action $T_X\times X\to X$ of $T_X$ on $X$ that extends the natural action of $T_X$ on itself.

A toric morphism is a morphism $f:X\to Y$ between toric varieties such that $f$ maps the torus $T_X\subset X$ into $T_Y\subset Y$ and $f|_{T_X}$ is a group homomorphism.

A toric contraction  is a contraction $f:X\to Y$ so that $f$ is a toric morphism. We remark that if $f:X\to Y$ is a toric contraction, then $X$ is of Fano type over $Y$.

A toric pair is a pair  $(X,B)$ so that $X$ is a toric variety and $B$ is invariant under the torus action of $T_X$.

We refer to \cite{Fu93} 
or \cite{CLS11} for the general theory of toric varieties. 

\subsection{Exceptional singularities}
\begin{defn}\label{defn:plt}
Let $(Z\ni z,\Delta)$ be a klt germ. A plt (resp. $\delta$-plt) blow-up of $(Z\ni z,\Delta)$ is a birational morphism $\pi:Y\to Z$ such that
\begin{itemize}
\item $Y$ has a prime divisor $E$ such that $\cent_Z E=z$ and $-E$ is ample over $Z$ (in particular, $\text{Ex}(\pi)=E$, where $\text{Ex}(\pi)$ is the exceptional locus of $\pi$);
\item $(Y,E+\Delta_Y)$ is plt (resp. $\delta$-plt) near $E$, where $\Delta_Y$ is the birational transform of $\Delta$ on $Y$.
\end{itemize}
\end{defn}

\begin{lem}\cite[Lemma 3.12]{HLS19}\label{lem:has}
Any klt germ $(Z\ni z,\Delta)$ with $\dim Z\geq 2$ has a plt blow-up.
\end{lem}

\begin{defn}\label{defn:exc}
Let $(Z\ni z,\Delta)$ be an lc germ. We say $(Z\ni z,\Delta)$ is an exceptional singularity, if for any effective $\R$-divisor $G$ on $Z$ so that $(Z\ni z,\Delta+G)$ is lc, there exists at most one prime divisor $E$ over $Z$ such that $a(E,Z,\Delta+G)=0$ and $z\in \cent_Z E$. Moreover, if  $(Z\ni z,\Delta)$ is a klt germ, we say $(Z\ni z,\Delta)$ is a klt exceptional singularity.
In the case when $\Delta=0$, we use  $Z\ni z$ instead of $(Z\ni z,0)$.
\end{defn}

\begin{lem}\cite[Lemma 3.16]{HLS19}\label{lem:unique}
Let $(Z\ni z,\Delta)$ be a klt exceptional singularity with $\dim Z\geq 2$. Then $(Z\ni z,\Delta)$ has a unique plt blow-up.
\end{lem}

\begin{lem}\cite[Lemma 3.22]{HLS19}\label{lem:exc}
Let $s\geq 2$ be an integer and $\mathfrak D\subset [0,1]$ be a DCC set. Then there exists a positive real number $\delta$ depending only on $s$ and $\mathfrak D$ satisfying the following. Let $(Z\ni z,\Delta)$ be a klt exceptional singularity so that $\dim Z=s$ and the coefficients of $\Delta$ are in $\mathfrak D$.
Let $f:Y\to Z$ be the unique plt blow-up of $(Z\ni z, \Delta)$. Then $f:Y\to Z$ is a $\delta$-plt blow-up of $(Z\ni z,\Delta)$.
\end{lem}

\section{Proof of (5) $\Rightarrow$  (3) in Theorem \ref{thm:main}}
The goal of this section is to show that Conjecture \ref{conj:birkar2} implies Conjecture \ref{conj:birkar1}. Theorem 1.1 in \cite{Bi23} plays a key role in the proof. We begin with
 a generalized version of Conjecture \ref{conj:birkar2}.
\begin{conj}\label{conj:generalized}
Let $s$ be a natural number and $\epsilon$ be a positive real number. Then there is a positive number $t$ depending only on $s,\epsilon$ satisfying the following. Assume $(Z,\Delta+\bN)$ is an $\epsilon$-lc g-pair of dimension $s$ with a closed point $z\in Z$. Then there exists an effective Cartier divisor $H$ on some neighbourhood of $z\in Z$ such that $z\in \Supp H$ and $(Z,\Delta+tH+\bN)$ is lc near $z$.
\end{conj}
\begin{prop}\label{prop:gene}
Assume Conjecture \ref{conj:birkar2} holds in dimension $\leq \ell$ for some natural number $\ell$. Then Conjecture \ref{conj:generalized} holds in dimension $\leq \ell$.
\end{prop}
\begin{proof}
Let $(Z,\Delta+\bN)$ be an $\epsilon$-lc g-pair of dimension $s\leq \ell$. Take a high enough birational model $\pi:Z'\to Z$ such that $\bN$ descends on $Z'$ and $\bN_{Z'}$ is nef over $Z$. Write 
$$K_{Z'}+\Delta'+\bN_{Z'}=\pi^*(K_Z+\Delta+\bN_{Z}).$$ Then $(Z',\Delta')$ is an $\epsilon$-lc sub-pair. Since $\bN_{Z'}$ is big and nef over $Z$, there exists an effective $\R$-divisor $E'$ on $Z'$ such that $E'\sim_{\R}\bN_{Z'}/Z$ and $(Z',\Delta'+E')$ is an $\frac{\epsilon}{2}$-lc sub-pair. As $K_{Z'}+\Delta'+E'\sim_{\R}0/Z$, we have
$$K_Z'+\Delta'+E'=\pi^*(K_Z+\Delta+E),$$
where $E=\pi_*E'$. So $(Z,\Delta+E)$ is an $\frac{\epsilon}{2}$-lc pair. By the assumption that Conjecture \ref{conj:birkar2} holds, for any closed point $z\in Z$, there exists an effective Cartier divisor $H$ on some neighbourhood of $z\in Z$ such that $z\in \Supp H$ and $(Z,\Delta+E+tH)$ is lc near $z$, where $t>0$ only depends on $s,\epsilon$. For any prime divisor $D$ over $Z$ with $z\in \cent_Z D$, we have
\begin{align*}
a(D,Z,\Delta+tH+\bN)&=a(D,Z',\Delta'+t\pi^*H)\geq a(D,Z',\Delta'+E'+t\pi^*H)\\
&=a(D,Z,\Delta+E+tH)\geq 0.
\end{align*}
Hence the g-pair $(Z,\Delta+tH+\bN)$ is lc near $z$.
\end{proof}

\begin{prop}\label{prop:5to3}
Assume Conjecture \ref{conj:birkar2} holds in dimension $\leq \ell$ for some natural number $\ell$. Then Conjecture \ref{conj:birkar1} holds when $\dim Z\leq \ell$.
\end{prop}
\begin{proof}
Let $(X,B)$ be a pair, $f:X\to Z$ be a contraction with $\dim Z\leq \ell$ and $z\in Z$ be a closed point satisfying the conditions in Conjecture \ref{conj:birkar1}. Then $-(K_X+B)$ is semi-ample over $Z$ as it is nef over $Z$ and as $X$ is of Fano type over $Z$. Hence there exists an effective $\R$-divisor $E$ on $X$ such that $E\sim_{\R} -(K_X+B)/Z$ and $(X,B+E)$ is still $\epsilon$-lc (here we may assume that $\epsilon<1$). Then $K_X+B+E\sim_{\R} 0/Z$. Let $(Z,\Delta+\bN)/Z$ be the g-pair induced by the canonical bundle formula for $f:(X,B+E)\to Z$. By \cite[Theorem 1.1]{Bi23}, $(Z,\Delta+\bN)$  is $\delta$-lc, where $\delta>0$ only depends on $d,\epsilon$. By Proposition \ref{prop:gene}, after shrinking $Z$ around $z$, there exists an effective Cartier divisor $H$ on $Z$ such that $z\in \Supp H$ and $(Z,\Delta+tH+\bN)$ is lc, where $t>0$ only depends on $\dim Z$ and $\delta$. Thus $t$ only depends on $d,\epsilon$ as $\dim Z\leq \dim X=d$ and as $\delta$ only depends on $d,\epsilon$. By Lemma \ref{lem:lc2}, $(X,B+E+tf^*H)$ is lc, which implies that $(X,B+tf^*H)$ is lc.
\end{proof}

\section{Proof of (3) $\Rightarrow$ (1) in Theorem \ref{thm:main}}
The goal of this section is to prove that Conjecture \ref{conj:birkar1} implies Conjecture \ref{conj:shokurov}. Many techniques used in this section come from \cite{Bi22}. We begin with a definition introduced by Birkar in \cite{Bi22}.

\begin{defn}
Let $d,r$ be natural numbers and $\epsilon$ be a positive real number. A $(d,r,\epsilon)$-Fano type fibration consists of a pair $(X,\Gamma)$ and  a contraction $f:X\to Z$ such that 
\begin{itemize}
  \item $(X,\Gamma)$ is  a projective $\epsilon$-lc pair of dimension $d$,
  \item $K_X+\Gamma\sim_{\R} f^*L$ for some $\R$-Cartier $\R$-divisor $L$ on $Z$,
  \item $-K_X$ is big over $Z$, i.e. $X$ is of Fano type over $Z$,
  \item $A$ is a very ample divisor on $Z$ with $A^{\dim Z}\leq r$, and 
  \item $A-L$ is ample.
\end{itemize}
\end{defn}

The following proposition is the same as \cite[Theorem 1.9]{Bi21} except that we remove the condition that $M-(K_X+B)$ is big. We follow the argument in the proof of \cite[Theorem 1.9]{Bi21} with small modifications.

\begin{prop}\label{prop:pp} Let $d,p$ be natural numbers. Then there exists a natural number $n$ depending only on $d,p$ satisfying the following. Assume 
\begin{itemize}
  \item $(X,B)$ is a projective lc pair of dimension $d$,
  \item $pB$ is integral,
  \item $f:X\to Z$ is a contraction,
  \item $M$ is the pullback of an ample Cartier divisor on $Z$,
  \item $X$ is of Fano type over $Z$,
  \item $M-(K_X+B)$ is nef, and 
  \item $S$ is a non-klt center of $(X,B)$ whose image on $Z$ is a closed point.
\end{itemize}
Then there is a $\Q$-divisor $\Lambda\geq B$ such that 
\begin{itemize}
  \item $(X,\Lambda)$ is lc over a neighbourhood of $z:=f(S)$, and 
  \item $n(K_X+\Lambda)\sim (n+2)M$.
\end{itemize}
\end{prop}

\begin{proof}
By the proof of \cite[Proposition 4.1]{Bi21}, Proposition \ref{prop:pp} holds under an additional assumption that there is a boundary $\Gamma$ on $X$ such that
\begin{itemize}
  \item $(X,\Gamma)$ is plt with $S=\lfloor \Gamma \rfloor$, and 
  \item $-(K_X+\Gamma)$ is ample over $Z$.
\end{itemize}
Note that the condition that $M-(K_X+B)$ is big has not been used in the proof of \cite[Proposition 4.1]{Bi21}.
So it suffices to construct a boundary $\Gamma$ satisfying the above conditions. 

Replacing $(X,B)$ with a $\Q$-factorial dlt model, we may assume that $X$ is $\Q$-factorial and $S$ is a component of $\lfloor B\rfloor$. The Fano type property of $X$ over $Z$ is preserved by Lemma \ref{lem:preserve} since $-(K_X+B)$ is nef over $Z$.
As $-K_X$ is big over $Z$ and as $S$ is vertical over $Z$, $-(K_X+S)$ is big over $Z$.

As $-(K_X+B)$ is nef over $Z$ and as $X$ is of Fano type over $Z$, $-(K_X+B)$ is semi-ample over $Z$.
Let $X\to U/Z$ be the contraction defining by $-(K_X+B)$ over $Z$. Then $K_X+B\sim_{\Q} 0/U$. As $-(K_X+S)$ is big over $Z$, it is also big over $U$. Run an MMP on $-(K_X+S)$ over $U$ and let $X'$ be the resulting model. In this process $S$ is not contracted because
$$B-S\sim_{\Q} -(K_X+S)/U$$
and because $S$ is not a component of $B-S$. Let $X'\to X''/U$ be the contraction defined by $-(K_{X'}+S')$ over $U$, where $S'$ is the pushdown of $S$ to $X'$. Then $X'\to X''$ is birational as $-(K_X+S)$ is big over $U$. We claim that $S'$ is not contracted over $X''$. Indeed, if not, since $K_{X'}+S'\sim_{\Q} 0/X''$, the pair $(X'',0)$ would not be klt. As $X''$ is of Fano type over $Z$ by Lemma \ref{lem:preserve}, there is a klt pair on $X''$, which leads to a contradiction.

Assume that there exist $n\in \N$ and $\Lambda''\geq B''$ such that $(X'',\Lambda'')$ is lc over $z$ and 
$$n(K_{X''}+\Lambda'')\sim (n+2)M'',$$
where $B'',M''$ are the pushdowns of $B,M$ to $X''$. Since $K_X+B\sim_{\Q} 0/U$, taking the crepant pullback of $K_{X''}+\Lambda''$ to $X$, we get $\Lambda\geq B$ such that $(X,\Lambda)$ is lc over $z$ and
$$n(K_{X}+\Lambda)\sim (n+2)M.$$
Thus we can replace $X$ with $X''$ and assume that $-(K_X+S)$ is ample over $U$. The $\Q$-factorial property of $X$ may be lost but we will not use it any more. Every non-klt center of $(X,S)$ is contained in $S$ because $X$ admits a klt pair as $X$ is of Fano type over $Z$.

Let $\Delta=(1-b)B+bS$ for a sufficiently small real number $b>0$. Then any non-klt place of $(X,\Delta)$ is also a non-klt place of $(X,S)$, and hence its center on $X$ is contained in $S$ and is mapped to $z$. Since $-(K_X+S)$ is ample over $U$ and $-(K_X+B)$ is $\Q$-linear equivalent to the pullback of an ample$/Z$ $\Q$-divisor on $U$, we have 
$$-(K_X+\Delta)=-(1-b)(K_X+B)-b(K_X+S)$$
is ample over $Z$.

By \cite[Lemma 2.7]{Bi21} applied to $(X,\Delta)$, there is a prime divisor $T$ over $X$ and a birational morphism $Y\to X$ such that
\begin{itemize}
  \item either $Y\to X$ is small or it contracts $T$ but no other divisors,
  \item $(Y,T)$ is plt,
  \item $-(K_Y+T)$ is ample over $X$, and
  \item $a(T,X,\Delta)=0$.
\end{itemize}
In particular, $T$ is mapped to $z$ since it is a non-klt place of $(X,\Delta)$. Since $\Delta \leq B$, we have $a(T,X,B)=0$.

Let $K_Y+\Delta_Y$ be the pullback of $K_X+\Delta$ to $Y$ and let
$$\Gamma=(1-c)\Delta_Y+cT$$
for some sufficiently small $c>0$. Since $-(K_X+\Delta)$ is ample over $Z$ and $-(K_Y+T)$ is ample over $X$, we have  
$$-(K_Y+\Gamma)=-(1-c)(K_Y+\Delta_Y)-c(K_Y+T)$$
is ample over $Z$.
Moreover, $(Y,\Gamma)$ is plt and $T=\lfloor \Gamma \rfloor$ is mapped to $z$.

Let $K_Y+B_Y$, $M_Y$ be the pullbacks of $K_X+B$, $M$ to $Y$. Replace $(X,B),M,S$ with $(Y,B_Y),M_Y,T$. Then there is a boundary $\Gamma$ on $X$ such that $(X,\Gamma)$ is plt with $S=\lfloor \Gamma \rfloor$ and $-(K_X+\Gamma)$ is ample over $Z$.
\end{proof}

The following proposition is the same as \cite[Proposition 4.15]{Bi22} except that we remove the condition that $-(K_X+B)$ is big over $Z$ and we emphasize that $n$ is independent of $r,\epsilon$. The notation is also slightly different; $B,\Gamma$ in the following proposition correspond to $\Delta, B$ in \cite[Proposition 4.15]{Bi22} respectively. We follow the argument in the proof of \cite[Proposition 4.15]{Bi22} with small modifications.

\begin{prop}\label{prop:lc}
Let $d,r$ be natural numbers, $\epsilon$ be a positive real number and $\mathfrak R\subset[0,1]$ be a finite set of rational numbers. Then there exist a natural number $n$ depending on $d,\mathfrak R$ and a natural number $m$ depending on $d,r,\epsilon,\mathfrak R$ satisfying the following (note that $n$ is independent of $r,\epsilon$). Assume that 
\begin{itemize}
  \item $(X,\Gamma)\to Z,A$ is a $(d,r,\epsilon)$-Fano type fibration, 
  \item $(X,B)$ is a pair with $B\leq \Gamma$, and
  \item the coefficients of $B$ are in $\mathfrak R$.
\end{itemize}
Then for any closed point $z\in Z$, there is a $\Q$-divisor $\Lambda$ on $X$ such that 
\begin{itemize}
  \item $\Lambda \geq B$,
  \item $(X,\Lambda)$ is lc over $z$, and
  \item $n(K_X+\Lambda)\sim mf^*A$.
\end{itemize}
\end{prop}
\begin{proof}

When $\dim Z=0$, the proposition follows from \cite[Theorem 1.7]{Bi19} (we first run an MMP on $-(K_X+B)$ to make $-(K_X+B)$ nef). So we may suppose that $\dim Z>0$.

Let $V_z$ be the sub-linear system of $|f^*A|$ consisting of elements containing the fiber $f^{-1}{z}$. Then $V_z$ is base point free outside the fiber $f^{-1}z$ as $A$ is very ample. Let $p$ be a natural number such that $\frac{1}{p}<1-\epsilon$ (we may assume that $\epsilon<1$). Take distinct general elements $M_1,\cdots,M_{p(d+1)}$ in $V_z$  and let
$$M=\frac{1}{p}(M_1+\cdots+M_{p(d+1)}).$$
Then $(X,\Gamma+M)$ is $\epsilon$-lc outside $f^{-1} z$ and is not lc at any point of $f^{-1} z$ by \cite[Theorem 18.22]{Ko92}.

By the ACC for lc thresholds \cite[Theorem 1.1]{HMX14}, there is a positive real number $e$ depending only on $d, \mathfrak R$ satisfying the following: for any pair $(Y,C)$ of dimension $d$ with the coefficients of $C$ in $\mathfrak R$ and for any prime divisor $D$ on $Y$ , if $(Y,C+(1-e)D)$ is lc, then $(Y,C+D)$ is lc.

Let $\epsilon'=\frac{1}{2}\min\{e,\epsilon\}$ and  let $u$ be the largest number such that $(X,\Gamma+uM)$ is $\epsilon'$-lc. Then there is a prime divisor $T$ over $X$ such that $a(T,X,\Gamma+uM)=\epsilon'$. Since $(X,\Gamma+M)$ is not lc near $f^{-1}{z}$, we have $u<1$. As $(X,\Gamma+uM)$ is $\epsilon$-lc outside $f^{-1}z$ and as $\epsilon'< \epsilon$, the center of $T$ on $X$ is contained in $f^{-1}z$. 
Note that 
$$K_X+\Gamma +uM\sim_{\R}f^*(L+u(d+1)A).$$
Replacing $\epsilon$ with $\epsilon'$ and replacing $\Gamma$ with $\Gamma+uM$ (and replacing $A,r$ with $(d+2)A$, $(d+2)^dr$ respectively),
we can assume that $\epsilon\leq e$ and that there exists a prime divisor $T$ over $X$ mapped to $z$ with $a(T,X,\Gamma)=\epsilon$. 

Let $\pi:\hat{X}\to X$ be a $\Q$-factorialisation extracting $T$ but no other divisors (if $T$ is not exceptional over $X$, then $\pi$ is a small $\Q$-factorialisation). Write $K_{\hat{X}}+\hat{\Gamma}=\pi^*(K_X+\Gamma)$. Let $\hat{\Theta}$ be the same as $\pi^*B$ except that we replace the coefficient of $T$ by 1. Then the coefficients of $\hat{\Theta}$ are in $\mathfrak R$ (we may suppose that $1\in\mathfrak R$).
By construction, $\hat{\Gamma}-\hat{\Theta}\geq -\epsilon T$, so
$$-(K_{\hat X}+\hat{\Theta})\geq -(K_{\hat X}+\hat{\Gamma})-\epsilon T.$$
Thus $-(K_{\hat X}+\hat{\Theta})$ is pesudo-effective over $Z$ as $K_{\hat X}+\hat{\Gamma}\sim_{\R} 0/Z$ and as $T$ is vertical over $Z$. Run an MMP on $-(K_{\hat{X}}+\hat{\Theta})$ over $Z$ and let $X'$ be the resulting model with the induced morphism $f':X'\to Z$. Denote the pushdowns of $\hat{\Gamma},\hat{\Theta},T$ to $X'$ by $\Gamma',\Theta',T'$. Then $-(K_{X'}+\Theta')$ is nef over $Z$. 
By construction, $\Gamma'-\Theta'\geq -\epsilon T'$, so $f'^*A+\Gamma'-\Theta'$ is pseudo-effective as $T$ is mapped to a closed point. By \cite[Proposition 4.2]{Bi22}, there is a natural number $l$ depending only on $d,r,\epsilon$ such that $lf'^*A-(K_{X'}+\Theta')$ is nef globally.

Since $(\hat{X},\hat{\Gamma})$ is $\epsilon$-lc and $K_{\hat{X}}+\hat{\Gamma}\sim_{\R}0/Z$, $(X',\Gamma')$ is also $\epsilon$-lc. Hence $(X',\Theta'-\epsilon T')$ is klt as $\Theta'-\epsilon T'\leq \Gamma'$. Since $\epsilon\leq e$, we deduce that $(X',\Theta')$ is lc. We claim that $T$ is not contracted over $X'$ (i.e. $T'\neq 0$). Indeed, if not, since $\hat{X}\dashrightarrow X'$ is an MMP on $-(K_{\hat{X}}+\hat{\Theta})$, we have
$$0=a(T,\hat{X},\hat{\Theta})> a(T,X',\Theta'),$$
which contradicts the fact that $(X',\Theta')$ is lc.

Applying Proposition \ref{prop:pp} (by taking $B=\Theta'$, $M=lf'^*A$ and $S=T'$), there exist $n\in \N$ depending only on $d,\mathfrak R$ and $\Lambda'\geq \Theta'$ such that $(X',\Lambda')$ is lc over $z$ and
$$n(K_{X'}+\Lambda')\sim (n+2)lf'^*A.$$
Since $\hat{X}\dashrightarrow X'$ is an MMP on $-(K_{\hat{X}}+\hat{\Theta})$, taking $K_{\hat{X}}+\hat{\Lambda}$ to be the crepant pullback of $K_{X'}+\Lambda'$ to $\hat{X}$, we get $\hat{\Lambda}\geq \hat{\Theta}$. Taking $\Lambda$ to be the pushdown of $\hat{\Lambda}$ to $X$, we get $\Lambda \geq B$. Moreover, $(X,\Lambda)$ is lc over $z$ and 
$$n(K_X+\Lambda)\sim (n+2)lf^*A.$$
Finally we take $m=(n+2)l$ and it depends only on $d,r,\epsilon,\mathfrak{R}$.
\end{proof}

The following theorem is similar to \cite[Theorem 1.7]{Bi22} but with three differences. First, we add the assumption that Conjecture \ref{conj:birkar1} holds. Second, we remove the condition that $-(K_X+B)$ is big over $Z$. Third, we make $n$ independent of $r$. The notation is also slightly different; $B,\Gamma$ in the following theorem correspond to $\Delta, B$ in \cite[Theorem 1.7]{Bi22} respectively. We follow the arguments in the proofs of \cite[Theorem 1.7 and Proposition 4.16]{Bi22} with some modifications.

\begin{thm}\label{thm:klt}
  Assume Conjecture \ref{conj:birkar1} holds when $\dim Z\leq \ell$ for some natural number $\ell$.
  Let $d,r$ be natural numbers, $\epsilon$ be a positive real number and $\mathfrak R\subset[0,1]$ be a finite set of rational numbers. Then there exist a natural number $n$ depending on $d,\epsilon,\mathfrak R$ and a natural number  $m$ depending on $d,r,\epsilon,\mathfrak R$ satisfying the following (note that $n$ is independent of $r$). Assume that 
  \begin{itemize}
    \item $(X,\Gamma)\to Z,A$ is a $(d,r,\epsilon)$-Fano type fibration with $\dim Z\leq \ell$, and 
    \item $(X,B)$ is a pair such that $B\leq \Gamma$ and the coefficients of $B$ are in $\mathfrak R$.
  \end{itemize}
  Then there is a $\Q$-divisor $\Lambda$ on $X$ such that 
  \begin{itemize}
    \item $\Lambda \geq B$,
    \item $(X,\Lambda)$ is klt, and
    \item $n(K_X+\Lambda)\sim mf^*A$.
  \end{itemize}
\end{thm}
\begin{proof}
\emph{Step 1.} When $\dim Z=0$,  $(X,B)$ is log bounded by \cite[Corollary 1.4]{Bi21}. Applying \cite[Lemma 2.4]{Bi19}, the Cartier index of $K_X+B$ depends only on $d,\epsilon,\mathfrak R$. 
By the effective base point freeness \cite{Ko93}, there is a natural number $n$ depending only on $d, \epsilon,\mathfrak R$ such that $-n(K_X+B)$ is base point free. Let $\Lambda=B+R$ where $nR$ is a general element of $|-n(K_X+B)|$. Then $n(K_X+\Lambda)\sim 0$ and $(X,\Lambda)$ is klt. So we may suppose that $\dim Z>0$.

\emph{Step 2.} It suffices to show that, for any closed point $z\in Z$, there is a boundary $\Lambda\geq B$ such that $(X,\Lambda)$ is klt over some neighbourhood $U_z$ of $z$ and $n(K_X+\Lambda)\sim mf^*A$, where $n$ only depends on $d,\epsilon,\mathfrak R$ and $m$ only depends on $d,r,\epsilon,\mathfrak R$. Indeed, if this is true, we can find finitely many closed points $z_1,\cdots,z_k$ in $Z$ such that the corresponding open sets $U_{z_i}$ cover $Z$. For each $z_i$, let $\Lambda_i$ be the corresponding boundary as above. Then
$$n(\Lambda_i-B)\in |mf^*A-n(K_X+B)|.$$
Therefore, if we let $\Lambda=B+R$ where $nR$ is a general element of $|mf^*A-n(K_X+B)|$, then $n(K_X+\Lambda)\sim mf^*A$ and $(X,\Lambda)$ is klt over each $U_{z_i}$. As the open sets $U_{z_i}$ cover $X$, it follows that $(X,\Lambda)$ is klt everywhere and we are done. From now on we fix  a  closed point $z\in Z$.

\emph{Step 3.} 
Let $V_z$ be the sub-linear system of $|A|$ consisting of elements containing $z$. Then $V_z$ satisfies conditions in Lemma \ref{lem:linearsystem}. 
By the assumption that Conjecture \ref{conj:birkar1} holds and by Corollary \ref{cor:idealdiv}, there is a natural number $e$ depending on $d,\epsilon$ such that $(X,\Gamma,f^*{\mathfrak m}_z^{2/e})$ is lc over $z$. It follows that $(X,\Gamma,f^*{\mathfrak m}_z^{2/e})$ is lc everywhere since $(X,\Gamma)$ is lc everywhere.
Applying Lemma \ref{lem:linearsystem}(3), there is an element $P\in V_z$ such that $(X,\Gamma+\frac{2}{e}f^*P)$ is lc.
It follows that  $(X,\Gamma+\frac{1}{e}f^*P)$ is $\frac{\epsilon}{2}$-lc. Moreover, we have 
$$K_X+\Gamma+\frac{1}{e} f^*P\sim_{\R} f^*(L+\frac{1}{e}A).$$
Replacing $\Gamma$ with $\Gamma+\frac{1}{e} f^*P$ and replacing $\epsilon$ with $\frac{\epsilon}{2}$ (and replacing $A,r$ with $2A,2^dr$ respectively), we may suppose that 
$$\Gamma\geq \widetilde{B}:=B+\frac{1}{e}f^*P.$$ 
Since $f^*P$ is integral and $e$ depends only on $d,\epsilon$, the coefficients of $\widetilde{B}$ belong to a finite set depending on $d,\epsilon,\mathfrak R$, so expanding $\mathfrak R$ we can assume they  belong to $\mathfrak R$.

\emph{Step 4.} In this step, we reduce the theorem to the existence of a special lc complement. Assume that there are $n\in \N$ depending on $d,\epsilon,\mathfrak R$ and $m\in \N$ depending on $d,r,\epsilon,\mathfrak R$ and there is a $\Q$-divisor $\Lambda$ on $X$ such that
\begin{enumerate}
  \item $\Lambda\geq \widetilde{B}$ 
  \item $(X,\Lambda)$ is lc over $z$,
  \item the non-klt locus of $(X,\Lambda)$ is mapped to a finite set of closed points of $Z$, and 
  \item $n(K_X+\Lambda)\sim mf^*A$.
\end{enumerate}
Let $Q$ be a general element of $|A|$ and let 
$$\Lambda':=\Lambda-\frac{1}{e}f^*P+\frac{1}{e}f^*Q.$$
By (2), any non-klt center of $(X,\Lambda)$ intersecting $f^{-1}z$ is contained in $f^{-1}z$. Thus $(X,\Lambda')$ is klt over $z$ since $f^*P$ contains $f^{-1}{z}$. Moreover, $\Lambda'\geq B$ and after replacing $n,m$ with $en,em$, we have 
$$n(K_X+\Lambda')=n(K_X+\Lambda-\frac{1}{e}f^*P+\frac{1}{e}f^*Q)\sim n(K_X+\Lambda)\sim mf^*A.$$
Therefore, it suffices to find $n,m,\Lambda$ satisfying (1)-(4). From now on we replace $B$ with $\widetilde{B}$.

\emph{Step 5.} After running an MMP on $-(K_X+B)$ over $Z$, we may suppose that $-(K_X+B)$ is nef over $Z$. Applying \cite[Proposition 4.2]{Bi22}, there is $l\in \N$ depending on $d,r,\epsilon$ such that $lf^*A-(K_X+B)$ is nef globally. By Proposition \ref{prop:lc}, there exist $n\in \N$ depending on $d,\mathfrak R$ and  $m\in \N$ depending on $d,\mathfrak R,\epsilon,r$ and there is a $\Q$-divisor $\Lambda\geq B$ such that $(X,\Lambda)$ is lc over $z$ and $n(K_X+\Lambda)\sim mf^*A$. Then 
$$n(\Lambda-B)\in |mf^*A-n(K_X+B)|.$$
Multiplying $n,m$ by a bounded natural number depending on $\mathfrak R$, we may suppose that $nB$ is integral.

By adding a general element of $|(l+2)f^*A|$ to $\Lambda$ and replacing $m$ by $m+2n+nl$ to preserve $n(K_X+\Lambda)\sim mf^*A$, we may assume that $m-3\geq nl$. Hence
$$(m-3)f^*A-n(K_X+B)$$
is nef.

\emph{Step 6. In this step we show that there is a boundary $\Delta$ on $X$ such that $(X,\Delta)$ is klt and $2f^*A-(K_X+\Delta)$ is ample.} Since $X$ is of Fano type over $Z$, there is a klt pair $(X,\Delta)$ such that $-(K_X+\Delta)$ is ample over $Z$.
Then there exists a real number $\alpha>0$ such that $\alpha f^*A-(K_X+\Delta)$ is ample. It follows that
\begin{align*}
&(1-t+t\alpha)f^*A-(K_X+(1-t)\Gamma+t\Delta)\\
=&(1-t)\big(f^*A-(K_X+\Gamma)\big)+t\big(\alpha f^*A-(K_X+\Delta)\big)
\end{align*}
is ample for any $t\in(0,1)$. So replacing $\Delta$ with $(1-t)\Gamma+t\Delta$ for some sufficiently small $t>0$, we can replace $\alpha$ with some number in $(0,2)$. Therefore 
$$2f^*A-(K_X+\Delta)$$ is ample. 

\emph{Step 7.} Let $H$ be a general element of $|A|$ and let $G=f^*H$. Denote by $g$ the induced morphism $G\to H$. Then 
\begin{align*}
&mf^*A-n(K_X+B)-G\\
\sim ~ & K_X+\Delta+\big(2f^*A-(K_X+\Delta)\big)+\big((m-3)f^*A-n(K_X+B)\big).
\end{align*}
Hence 
$$H^1(X, mf^*A-n(K_X+B)-G)=0$$
by the Kawamata-Viehweg vanishing theorem, which implies that the restriction map
$$H^0\big(X,mf^*A-n(K_X+B)\big)\to H^0\big(G,(mf^*A-n(K_X+B))|_G\big)$$
is surjective. Note that for any integral $\Q$-Cartier divisor $D$ on $X$, by \cite[Lemma 2.42]{Bi19} and by the choice of $G$ we have 
$$\mathcal O_X(D)\otimes \mathcal O_G\cong \mathcal O_G(D|_G),$$
which is used to obtain the above surjectivity (see \cite[2.41]{Bi19}).

\emph{Step 8.} Define 
$$K_G+\Gamma_G=(K_X+\Gamma+G)|_G$$
and
$$K_G+B_G=(K_X+B+G)|_G.$$
By the choice of $G$, we have $\Gamma_G=\Gamma|_G$ and $B_G=B|_G$. Then
\begin{itemize}
  \item $(G,\Gamma_G)$ is $\epsilon$-lc as $(X,\Gamma)$ is so,
  \item $-K_G$ is big over $H$ as $-K_G=-(K_X+G)|_G\sim -K_X|_G$ over $H$,
  \item $K_G+\Gamma_G\sim_{\R} (f^*L+G)|_G\sim_{\R} g^*(L+A)|_H$, and
  \item $2A|_H-(L+A)|_H$ is ample as $A-L$ is ample.
\end{itemize}
Thus $g:(G,\Gamma_G)\to H,2A|_H$ is a $(d-1,2^{d-1}r,\epsilon)$-Fano type fibration. Moreover, $B_G\leq \Gamma_G$ and the coefficients of $B_G$ are in $\mathfrak R$. Since $\dim G=d-1$, by induction there exist $p\in \N$ depending on $d,\epsilon,\mathfrak R,$ and $q\in \N$ depending on $d,r,\epsilon,\mathfrak R$ and there is a $\Q$-divisor $\Lambda'_G\geq B_G$ such that $(G,\Lambda'_G)$ is klt and 
$$p(K_G+\Lambda'_G)\sim q g^*A|_H.$$
Replacing both $n$, $p$ with $np$ and replacing $m$, $q$ with $mp$, $nq$ respectively, we may suppose that $n=p$. If $q<m+n$, we increase $q$ to $m+n$ by adding $\frac{1}{p}D_G$ to $\Lambda'_G$ where $D_G$ is a general element of $|(m+n-q)g^*A|_H|$; if $q>m+n$, we increase $m$ to $q-n$ by adding $\frac{1}{n}D$ to $\Lambda$ where $D$ is a general element of $|(q-n-m)f^*A|$. So we may suppose that $q=m+n$. Thus now we have
$$n(K_G+\Lambda'_G)\sim (m+n)g^*A|_H.$$
Note that in this process, the inequality $m-3\geq nl$ in Step 5 is preserved. Hence the surjectivity in Step 7 still holds.

\emph{Step 9.} By construction,
$$n(\Lambda'_G-B_G)\in |(m+n)g^*A|_H-n(K_G+B_G)|$$
and the pair $(G,B_G+\Lambda'_G-B_G)$
is klt. Therefore, if $nR_G$ is a general element of
$$|(m+n)g^*A|_H-n(K_G+B_G)|,$$
the pair $(G,B_G+R_G)$ is klt. On the other hand,
\begin{align*}
  (m+n)g^*A|_H-n(K_G+B_G)&=\big((m+n)f^*A-n(K_X+B+G)\big)|_G\\
  &\sim\big(mf^*A-n(K_X+B)\big)|_G.
\end{align*}
Let $nR$ be  a general element  of $|mf^*A-n(K_X+B)|$. By the surjectivity in Step 7, it restricts to a general element 
$$nR_G\in|(m+n)g^*A|_H-n(K_G+B_G)|.$$
Note that
$$K_G+B_G+R_G=(K_X+B+R+G)|_G,$$
by inversion of adjunction \cite[Theorem 5.50]{KM98}, we deduce that $(X,B+R+G)$ is plt near $G$, which implies that $(X,B+R)$ is klt near $G$. Since $G$ is an ample divisor on $Z$, the non-klt locus of $(X,B+R)$ is mapped to  a finite set of closed points of $Z$. Note that $(X,B+R)$ is lc over $z$ because $(X,\Lambda)$ is lc over $z$, $n(\Lambda -B)\in |mf^*A-n(K_X+B)|$ and $nR$ is a general element of $|mf^*A-n(K_X+B)|$. Replacing $\Lambda$ by $B+R$, the conditions in Step 4 are satisfied and we can finish the proof.
\end{proof}

\begin{lem}\label{lem:projective}
Let $f:X\to Z$ be a contraction such that $X$ is of Fano type over $Z$ and $X$ is $\Q$-factorial. Let $(X,B)$ be an $\epsilon$-lc pair with $0< \epsilon<1$ such that $-(K_X+B)$ is nef over $Z$ and the coefficients of $B$ are in a set $\mathfrak R$ of rational numbers. Then there exist a pair $(\bar{X},\bar{B})$ and a contraction $\bar{f}:\bar{X}\to \bar{Z}$ such that
\begin{itemize}
  \item $\bar{X}, \bar{Z}$ are projective,
  \item $\bar{X}$ is $\Q$-factorial,
  \item $\bar{X}$ is of Fano type over $\bar{Z}$
  \item $(\bar{X},\bar{B})$ is $\epsilon$-lc, 
  \item the coefficients of $\bar{B}$ are in $\mathfrak R$,
  \item $-(K_{\bar{X}}+\bar{B})$ is nef over $\bar{Z}$,
  \item $Z$ is an open subset of $\bar{Z}$, and 
  \item $(X,B)=(\bar{X},\bar{B})\times_{\bar{Z}} Z$.
\end{itemize}
\end{lem}
\begin{proof}
As $X$ is of Fano type over $Z$ and as $-(K_X+B)$ is nef over $Z$, $-(K_X+B)$ is semi-ample over $Z$. Thus there is a $\Q$-boundary $\Gamma\geq B$ such that $(X,\Gamma)$ is $\epsilon$-lc and $K_X+\Gamma\sim_{\Q} 0/Z$. 

Since $X,Z$ are assumed to be quasi-projective, we can find compactifications $X\subset X^{c}$ and $Z\subset Z^{c}$ such that $X^c,Z^c$ are projective and the morphism $f:X\to Z$ can be extended to $f^c:X^c\to Z^c$. Let $\Gamma^c$ be the closure of $\Gamma$ in $X^c$ and let $\pi^c:W^c\to X^c$ be a log resolution of $(X^c,\Gamma^c\cup (f^c)^{-1}(Z^c\setminus Z))$. Denote $W=(\pi^c)^{-1}(X)$ and $\pi=\pi^c\mid_W$. Write
$$K_W+\Gamma_W-F_W=\pi^*(K_X+\Gamma)$$
where $\Gamma_W,F_W$ are effective $\Q$-divisors without common components.  Let $\Gamma_W^c$ be the closure of $\Gamma_W$ in $W^c$. Then $(W^c,\Gamma_W^c)$ is an $\epsilon$-lc pair. By \cite{BCHM}, $(W^c,\Gamma_W^c)$ has a log canonical model over $X^c$, which we denote by $(\hat{X},\hat{\Gamma})$. Then $(\hat{X},\hat{\Gamma})$ is an $\epsilon$-lc pair and $(\hat{X},\hat{\Gamma})\times_{Z^{c}} Z=(X,\Gamma)$. Let $\widetilde{X}$ be a small $\Q$-factorialization of $\hat{X}$. Since $X$ is $\Q$-factorial, we have $\widetilde{X}\times_{Z^c} Z=X$. Replacing $\hat{X}$ with $\widetilde{X}$, we may assume that $\hat{X}$ is $\Q$-factorial.
  
By \cite[Theorem 5.2]{Bi12}, we can run an MMP on $K_{\hat{X}}+\hat{\Gamma}$ over $Z^c$ and the MMP ends with a good minimal model over $Z^c$, which we denote by $(\bar{X},\bar{\Gamma})$. Then $K_{\bar{X}}+\bar{\Gamma}$ is semi-ample over $Z^c$. Moreover, $(\bar{X},\bar{\Gamma})$ is $\epsilon$-lc and $(\bar{X},\bar{\Gamma})\times_{Z^{c}} Z=(X,\Gamma)$. Let $\bar{f}:\bar{X}\to \bar{Z}/Z^c$ be the contraction induced by $K_{\bar{X}}+\bar{\Gamma}$ over $Z^c$. Then $\bar{Z}\times_{Z^c} Z=Z$ and the restriction of $\bar{f}$ over $Z$ coincides with $f:X\to Z$. 
By now we have constructed a contraction $\bar{X}\to \bar{Z}$ between projective varieties and an $\epsilon$-lc pair $(\bar{X},\bar{\Gamma})$ such that $\bar{X}$ is $\Q$-factorial, $Z$ is an open subset of $\bar{Z}$, $K_{\bar{X}}+\bar{\Gamma}\sim_{\Q}0/\bar{Z}$ and $(\bar{X},\bar{\Gamma})\times_{\bar{Z}} Z=(X,\Gamma)$. As $-K_X$ is big over $Z$, $-K_{\bar{X}}$ is also big over $\bar{Z}$, which implies that $\bar{X}$ is of Fano type over $\bar{Z}$.

Denote by $\bar{B}$ the closure of $B$ in $\bar{X}$. Then the coefficients of $\bar{B}$ are in $\mathfrak R$ and $\bar{B}\leq \bar{\Gamma}$. Run an MMP on $-(K_{\bar{X}}+\bar{B})$ over $\bar{Z}$ and let $X'$ be the resulting model. Then $X'\times_{\bar{Z}} Z=X$ as $-(K_X+B)$ is nef over $Z$. Denote by $B',\Gamma'$ the pushdowns of $B,\Gamma$ to $X'$. Then $-(K_{X'}+B')$ is nef over $\bar{Z}$. Moreover, $(X',\Gamma')$ is $\epsilon$-lc as $(\bar{X},\bar{\Gamma})$ is $\epsilon$-lc  and as $K_{\bar{X}}+\bar{\Gamma}\sim_{\Q} 0/\bar{Z}$, which implies that $(X',B')$ is also $\epsilon$-lc. Finally we replace $(\bar{X},\bar{B})$ by $(X',B')$.
\end{proof}

\begin{thm}\label{thm:tech}
  Assume Conjecture \ref{conj:birkar1} holds when $\dim Z\leq \ell$ for some natural number $\ell$. Let $d$ be a natural number, $\epsilon$ be a positive real number and $\mathfrak R\subset [0,1]$ be a finite set of rational numbers. Then there is a natural number $n$ depending only on $d,\epsilon,\mathfrak R$ satisfying the following. Assume $(X,B)$ is a pair and $X\to Z$ is a contraction such that 
\begin{itemize}
\item $(X,B)$ is $\epsilon$-lc of dimension $d$ and $\dim Z\leq \ell$,
\item the coefficients of $B$ are in $\mathfrak R$,
\item $X$ is of Fano type over $Z$, and 
\item $-(K_X+B)$ is nef over $Z$.  
\end{itemize}
Then there is a $\Q$-divisor $\Lambda$ on $X$ such that 
  \begin{itemize}
    \item $\Lambda \geq B$,
    \item $(X,\Lambda)$ is klt, and
    \item $n(K_X+\Lambda)\sim 0/Z$.
  \end{itemize}
\end{thm}
\begin{proof}
We may assume that $\epsilon<1$. Taking a small $\Q$-factorialization, we may assume that $X$ is $\Q$-factorial. Applying Lemma \ref{lem:projective}, we can assume that $X,Z$ are projective. As $X$ is of Fano type over $Z$ and as $-(K_X+B)$ is nef over $Z$,  $-(K_X+B)$ is semi-ample over $Z$. 
Thus there is a $\Q$-boundary $\Gamma\geq B$ such that $(X,\Gamma)$ is $\epsilon$-lc and $K_X+\Gamma\sim_{\Q} 0/Z$. That is, $K_X+\Gamma\sim_{\Q}f^*L$ for some $\Q$-Cartier $\Q$-divisor $L$ on $Z$. Let $A$ be a very ample divisor on $Z$ such that $A-L$ is ample and denote $r=A^{\dim Z}$ (note that $r$ may not depend on $d,\epsilon, \mathfrak R$). 
Then $f:(X,\Gamma)\to Z, A$ is a $(d,r,\epsilon)$-Fano type fibration.
By Theorem \ref{thm:klt}, there is $n\in \N$ depending on $d,\epsilon, \mathfrak R$ and $m\in \N$ depending on $d,r,\epsilon,\mathfrak R$ and there is a boundary $\Lambda\geq B$ such that $(X,\Lambda)$ is klt and $n(K_X+\Lambda)\sim mf^*A$. Thus $n(K_X+\Lambda)\sim 0/Z$.
\end{proof}

\begin{cor}\label{cor:3to1}
Assume Conjecture \ref{conj:birkar1} holds when $\dim Z\leq \ell$ for some natural number $\ell$. Then Conjecture \ref{conj:shokurov} holds when $\dim Z\leq \ell$.
\end{cor}
\begin{proof}
It follows from Theorem \ref{thm:tech} directly.
\end{proof}

\section{Proof of (2) $\Rightarrow$  (4) in Theorem \ref{thm:main}}

\begin{prop}\label{prop:2to4}
Assume Conjecture \ref{conj:shokurov} holds when $X\to Z$ is birational, $B=0$ and $\dim Z\leq \ell$ for some natural number $\ell$. Then Conjecture \ref{conj:birkar1} holds when $X\to Z$ is birational, $B=0$ and $\dim Z\leq \ell$.
\end{prop}
\begin{proof}
Let $f:X\to Z$ be a birational contraction with $\dim X=d$ and $\dim Z\leq \ell$ such that $X$ is of Fano type over $Z$, $X$ is $\epsilon$-lc and $-K_X$ is nef over $Z$. Let $z\in Z$ be a closed point. By the assumption that Conjecture \ref{conj:shokurov} holds in this special case, there is a natural number $n$ depending only on $d,\epsilon$ such that, after shrinking $Z$ around $z$ there is a boundary $\Lambda_X$ on $X$ such that $(X,\Lambda_X)$ is klt and $n(K_X+\Lambda_X)\sim 0/Z$. Let $\Lambda$ be the pushdown of $\Lambda_X$ on $Z$. Then $(Z,\Lambda)$ is klt and $n(K_Z+\Lambda)$ is Cartier. It suffices to show that there are $t>0$ depending only on $d,\epsilon$ and an effective Cartier divisor $H$ on a neighbourhood of $z\in Z$ such that $z\in \Supp H$ and $(Z,\Lambda+tH)$ is lc near $z$. Indeed, if $(Z,\Lambda+tH)$ is lc near $z$, $(X,\Lambda_X+tf^*H)$ is lc over $z$, which implies that $(X,tf^*H)$ is so.

Let $\pi:Y\to Z$ be a plt blow-up of $(Z\ni z,\Lambda)$ with the exceptional divisor $E$ (its existence is guaranteed by Lemma \ref{lem:has}). Then $(Y,E+\Lambda_Y)$ is plt and $-E$ is ample over $Z$, where $\Lambda_Y$ is the birational transform of $\Lambda$ on $Y$. Then
$$K_Y+E+\Lambda_Y=\pi^*(K_Z+\Lambda)+a(E,Z,\Lambda)\cdot E,$$
which implies that $-(K_Y+E+\Lambda_Y)$ is ample over $Z$.
We claim that $Y$ is of Fano type over $Z$. Indeed, $(Y,(1-e)E+\Lambda_Y)$ is klt and $-(K_Y+(1-e)E+\Lambda_Y)$ is ample over $Z$ for a sufficiently small $e>0$. Moreover, $n(E+\Lambda_Y)$ is an integral divisor. By the boundedness of lc complements \cite[Theorem 1.8]{Bi19}, there is a natural number $m$ depending only on $d,n$ such that, after shrinking $Z$ around $z$ there is a boundary $\Omega_Y\geq E+\Lambda_Y$ on $Y$ such that $(Y,\Omega_Y)$ is lc and $m(K_Y+\Omega_Y)\sim 0/Z$. Let $\Omega$ be the pushdown of $\Omega_Y$ on $Z$. By construction, we have $(Z,\Omega)$ is lc, $m(K_Z+\Omega)$ is Cartier, $\Omega\geq \Lambda$, and $(Z,\Omega)$ is not klt near $z$. Thus $z\in \Supp (\Omega-\Lambda)$. Let $H=mn(\Omega-\Lambda)$ and $t=\frac{1}{mn}$, then $H$ is Cartier and $(Z,\Lambda+tH)$ is lc.
\end{proof}

\section{Proof of (4) $\Rightarrow$  (5) in Theorem \ref{thm:main}}
\begin{prop}\label{prop:4to5}
Assume that Conjecture \ref{conj:birkar1} holds when $X\to Z$ is birational, $B=0$ and $\dim Z\leq \ell$ for some natural number $\ell$. Then Conjecture \ref{conj:birkar2} holds when $\dim Z\leq \ell$.
\end{prop}
\begin{proof}
Let $(Z,\Delta)$ be an $\epsilon$-lc pair with a closed point $z\in Z$. We may assume that $\epsilon<1$. Denote by $\mathfrak m_z$ the ideal sheaf of $z\in Z$. Let $t$ be largest number such that the triple $(Z,\Delta,\mathfrak m_z^t)$ is $\frac{\epsilon}{2}$-lc.  By Corollary \ref{cor:idealdiv}, it suffices to show that $t$ is bounded from below by a positive constant depending only on $\dim Z, \epsilon$. If $t>1-\epsilon$, we are done. So we may assume that $t\leq 1-\epsilon$.

There exists a prime divisor $T$ over $Z$ such that 
$$a(T,Z,\Delta,\mathfrak m_{z}^t)=\frac{\epsilon}{2}.$$
Since $(Z,\Delta)$ is $\epsilon$-lc, we have
$$t \cdot \mu_T \mathfrak m_{z}\geq \frac{\epsilon}{2}.$$
So it suffices to show that $\mu_T \mathfrak m_{z}$ is bounded from above by a constant depending only on $\dim Z, \epsilon$.

Let $f:X\to Z$ be a $\Q$-factorialisation extracting $T$ but no other divisors (if $T$ is not exceptional over $Z$, then $f$ is a small $\Q$-factorialisation). As $(Z,\Delta,\mathfrak m_z^t)$ is $\frac{\epsilon}{2}$-lc and as $a(T,Z,\Delta,\mathfrak m_z^t)=\frac{\epsilon}{2}$, by Lemma \ref{lem:linearsystem} and Remark \ref{rem}, 
there exists an effective Cartier divisor $D$ on $Z$ such that $z\in \Supp D$, $(Z,\Delta+tD)$ is $\frac{\epsilon}{2}$-lc and
$$a(T,Z,\Delta+tD)=\frac{\epsilon}{2}.$$
Let $K_X+\Gamma$ be the pullback of $K_Z+\Delta+tD$ to $X$. Then $\Gamma$ is effective as $\mu_T \Gamma=\frac{\epsilon}{2}\geq 0$. Moreover, $(X,\Gamma)$ is $\frac{\epsilon}{2}$-lc and $K_X+\Gamma\sim_{\R} 0/Z$. Hence $X$ is of Fano type over $Z$.

Run an MMP on $-K_X$ over $Z$ and let $X'$ be the resulting model with the induced morphism $f':X'\to Z$. Then $-K_{X'}$ is nef over $Z$. As $(X,\Gamma)$ is $\frac{\epsilon}{2}$-lc and as $K_X+\Gamma\sim_{\R} 0/Z$, $(X',\Gamma')$ is also $\frac{\epsilon}{2}$-lc, where $\Gamma'$ is the pushdown of $\Gamma$ to $X'$. Thus $X'$ is $\frac{\epsilon}{2}$-lc. By  assumption Conjecture \ref{conj:birkar1} holds for $X'\to Z$, that is, there are a positive number $u$ depending on $\dim Z,\epsilon$ and an effective Cartier divisor $F$ on some neighbourhood of $z\in Z$ such that $z\in \Supp F$ and $(X',u\cdot f'^*F)$ is lc over $z$. Then $(X,u\cdot f^*F)$ is lc over $z$, as $X\dashrightarrow X'$ is an MMP on $-K_X$. Thus 
$$\mu_T \mathfrak m_{z}\leq \mu_T f^*F\leq \frac{1}{u},$$
where $u$ only depends on $\dim Z,\epsilon$.
\end{proof}
\begin{proof}[Proof of Theorem \ref{thm:main}]
It follows from Proposition \ref{prop:5to3}, \ref{prop:2to4}, \ref{prop:4to5} and Corollary \ref{cor:3to1}.
\end{proof}

\section{The surface case}
The goal of this section is to prove Conjecture \ref{conj:birkar2} in dimension 2. We begin with some properties of klt germs.
\begin{lem}\cite[Theorem 5.22]{KM98}\label{lem:rational}
  A klt germ $Z\ni z$ is a rational singularity.
\end{lem}
  
  \begin{lem}
  \cite[Theorem 4-6-18(iii)]{Ma02}\label{lem:kltquo}
  A klt germ $Z\ni z$ of dimension 2 is analytically isomorphic to $(\mathbb C^2\ni 0)/G$, where $G\subset GL(2,\C)$ is a finite subgroup.
  \end{lem}
  
  \begin{prop}\label{prop:6E}
    Let $Z\ni z$ be a klt germ of dimension 2 and let $\pi:Y\to Z$ be the minimal resolution with the exceptional divisor $E$. Denote by $C_f$ the fundamental cycle of the singularity $Z\ni z$ on $Y$. Then $C_f\leq 6E$.
  \end{prop}
\begin{proof}
  Note that the fundamental cycle depends only on the analytic isomorphism class  of the singularity, as it is determined by the intersection matrix of components of the exceptional divisor. The proposition follows from Lemma \ref{lem:kltquo} and Proposition \ref{prop:quotient} in the appendix.
\end{proof}

\begin{lem}\label{lem:klt}
Let $(Z,\Delta)$ be a pair of dimension 2 with a closed point  $z\in Z$ such that $\mld(Z\ni z,\Delta)>0$. Then $K_Z$ is $\Q$-Cartier near $z$ and $Z$ is klt near $z$.
\end{lem}
\begin{proof}
Let $\pi:W\to Z$ be the minimal resolution. Shrinking $Z$ we may assume that $\pi$ is an isomorphism outside $z$. Write 
$$K_W+\Delta_W+E_W=\pi^*(K_Z+\Delta)$$ 
where $\Delta_W$ is the birational transform of $\Delta$ and $E_W$ is exceptional over $Z$.

We claim that $(W,E_W)$ is klt over $z$. Indeed, for any prime divisor $D$ over $W$ with $\pi(\cent_W D)=z$, we have
$$a(D,W,E_W)\geq a(D,W,\Delta_W+E_W)=a(D,Z,\Delta)>0;$$
for any prime divisor $D$ over $W$ so that $\pi(\cent_W D)$ contains $z$ but is not $z$, we have $$a(D,W,E_W)=1$$
because $D$ is a prime divisor on $W$ which is not exceptional over $Z$. Thus $(W,E_W)$ is klt over $z$. Shrinking $Z$ around $z$, we may suppose that $(W,E_W)$ is klt everywhere.

For any exceptional$/Z$ prime divisor $T$ on $W$, we have 
$$-(K_W+E_W)\cdot T=\Delta_W\cdot T\geq 0.$$
So $-(K_W+E_W)$ is big and nef over $Z$. Therefore, $W$ is of Fano type over $Z$. 

Denote by $T_1,\cdots,T_k$ all the exceptional$/Z$ prime divisors on $W$. Then the intersection matrix $(T_i\cdot T_j)$ is negative definite. So there are rational numbers $a_1,\cdots,a_k$ such that
$$(K_W+T_W)\cdot T_j=0 \quad \text{ where } T_W=\sum_{i=1}^k a_iT_i,$$
for $j=1,\cdots,k$. Hence $K_W+T_W\equiv 0/Z$, which implies that $K_W+T_W\sim_{\Q} 0/Z$ as $W$ is of Fano type over $Z$. Thus $K_Z$ is $\Q$-Cartier as it is the pushdown of $K_W+T_W$ to $Z$.

Finally we show that $Z$ is klt. Since $W$ is of Fano type over $Z$, there is a klt pair $(W,\Gamma_W)$ such that $K_W+\Gamma_W\sim_{\R} 0/Z$. Let $\Gamma$ be the pushdown of $\Gamma_W$ to $Z$. Then $(Z,\Gamma)$ is klt, which implies that $Z$ is klt. 
\end{proof}

We need the following result in \cite{Ch22} on the lower bound of the lc threshold of a smooth curve with respect to a pair on a smooth surface.
\begin{lem}\label{lem:lct}\cite[Theorem 1.9]{Ch22}
Let $Z$ be a smooth surface with a closed point $z\in Z$ and let $(Z,\Delta)$ be a pair such that
$\mld(Z\ni z,\Delta)\geq \epsilon$ with $0<\epsilon\leq 1$. Let $H$ be a smooth curve on $Z$ passing $z$ such that $\mu_H \Delta \leq 1-\epsilon$ and $(\Delta'\cdot H)_z\leq 2$, where $\Delta'=\Delta-\mu_H \Delta \cdot H$. Then 
$(Z,\Delta+\frac{\epsilon^2}{2}H)$
is lc near $z$.
\end{lem}

\begin{proof}[Proof of Theorem \ref{thm:surface}]

\emph{Case 1:  $z\in Z$ is a smooth point.} Since $(Z,\Delta)$ is lc near $z$, the multiplicity of $\Delta$ at $z$ is at most 2. Hence there is a smooth curve $H$ on some neighbourhood of $z\in Z$ such that $z\in H$, $H \not\subset \Supp \Delta$ and $(\Delta\cdot H)_{z}\leq 2$. Applying Lemma \ref{lem:lct}, we deduce that $(Z,\Delta+\frac{\epsilon^2}{2}H)$ is lc near $z$.

\emph{Case 2: $z\in Z$ is not a smooth point.} Let $\pi:W\to Z$ be the   minimal resolution and denote by  $E$ the exceptional divisor on $W$ over $Z\ni z$. 
Let $K_W+\Delta_W$ be the pullback of $K_Z+\Delta$ to $W$. Then $\Delta_W\geq 0$ as $\pi$ is the minimal resolution.  

By Corollary \ref{cor:idealdiv}, it suffices to show that 
$(W, \Delta_W, \pi^*\mathfrak m_z^t)$
is lc over $z$, where $t=\epsilon^2/24$ and 
$\mathfrak m_z$ is the ideal sheaf of $z\in Z$. By Lemma \ref{lem:klt}, $Z\ni z$ is a klt germ and hence is a rational singularity by Lemma \ref{lem:rational}. It follows from Lemma \ref{lem:good} that $E$ is a simple normal crossing divisor on $W$. By Lemma \ref{lem:fundamental}, 
$$\pi^*\mathfrak m_{z}=\mathcal O_W(-C_f)$$
where $C_f$ is the fundamental cycle of the singularity $Z\ni z$ on $W$. Note that $C_f\leq 6E$ by Proposition \ref{prop:6E}, so it suffices to show that $(W,\Delta_W+\frac{\epsilon^2}{4}E)$ is lc near $E$.

Let $D$ be an irreducible component of $E$. Then $D$ is a smooth rational curve by Lemma \ref{lem:good}, which implies that $(K_W+D)\cdot D =-2$. Since $(K_W+\Delta_W)\cdot D=0$, we have 
$$(\Delta_W-D)\cdot D=2.$$
Denote $a=\mu_D \Delta_W$. Then $a\leq 1-\epsilon<1$ as $\mld(Z\ni z,\Delta)\geq \epsilon$. Note that $D\cdot D<0$, we have 
$$(\Delta_W-aD)\cdot D<2.$$
Applying Lemma \ref{lem:lct}, we deduce that 
$(W,\Delta_W+\frac{\epsilon^2}{2} D)$ is lc near $D$.

It suffices to show that $(W,\Delta_W+\frac{\epsilon^2}{4}E)$ is lc near $p$ for any closed point $p\in E$. If $p$ is not the intersection point of two components of $E$, by the above argument, $(W,\Delta_W+\frac{\epsilon^2}{2} E)$ is lc near $p$. Otherwise, $p$ is the intersection point of two components $D,F$ of $E$, then $E=D+F$ near $p$. By the above argument, $(W,\Delta_W+\frac{\epsilon^2}{2} D)$ and $(W,\Delta_W+\frac{\epsilon^2}{2} F)$ are lc near $p$, which implies that 
$$\big(W,\Delta_W+\frac{\epsilon^2}{4} (D+F)\big)$$ is lc near $p$. Therefore, $(W,\Delta_W+\frac{\epsilon^2}{4}E)$ is lc near $p$.
\end{proof}

\begin{proof}[Proof of Corollary \ref{cor:dim2}]
When $\dim Z=1$, Conjecture \ref{conj:birkar2} holds obviously. When $\dim Z=2$, Conjecture \ref{conj:birkar2} follows from Theorem \ref{thm:surface} directly. By Theorem \ref{thm:main}, Conjecture \ref{conj:shokurov} and \ref{conj:birkar1} hold when $\dim Z\leq 2$.
\end{proof}

\section{The toric case}
The goal of this section is to prove Conjecture \ref{conj:birkar1} and \ref{conj:birkar2} in the toric case. First  we recall a result of Ambro \cite[Theorem 3.12]{Am22}. Note that the condition in \cite[Theorem 3.12]{Am22} that $-(K_X+B)$ is semi-ample over $Z$ is equivalent to that   $-(K_X+B)$ is nef over $Z$, because $X$ is of Fano type over $Z$ as $X\to Z$ is a toric contraction. 
\begin{lem}\cite[Theorem 3.12]{Am22}\label{lem:Ambro}
Let $d$ be a natural number and $\mathfrak A\subset [0,1]$ be an ACC set. Then there is a positive real number $t$ depending only on $d,\mathfrak A$ satisfying the following. Assume $(X,B)$ is a toric pair, $Z$ is an affine toric variety with $\dim Z>0$, $f:X\to Z$ is a toric contraction and $z\in Z$ is a torus invariant closed point such that 
\begin{itemize}
  \item $\mld(X/Z\ni z,B)\geq \epsilon$ for some $\epsilon>0$, 
  \item $B=\sum_i(1-a_i)B_i$ with $\frac{a_i}{\epsilon}\in \mathfrak A$, where the $B_i$ are distinct toric prime divisors,
  \item $-(K_X+B)$ is nef over $Z$. 
\end{itemize}
Then there exists an effective toric Cartier divisor $H$ on $Z$ such that $z\in \Supp H$ and $(X,B+\gamma f^*H)$ is lc over $z$, where $\gamma=\epsilon\cdot t$.
\end{lem}
The following lemma is a direct consequence of Lemma \ref{lem:Ambro}, which asserts that Conjecture \ref{conj:birkar1} holds in the toric case when the coefficients of $B$ belong to a fixed DCC set.  
\begin{lem}\label{lem:bu}
Let $d$ be a natural number, $\epsilon$ be a positive real number and $\mathfrak D\subset [0,1]$ be a DCC set. Then there is a positive number $t$ depending only on $d,\epsilon, \mathfrak D$ satisfying the following.  Assume $(X,B)$ is a toric pair, $Z$ is an affine toric variety with $\dim Z>0$, $f:X\to Z$ is a toric contraction and $z\in Z$ is a torus invariant closed point such that
\begin{itemize}
  \item $\mld(X/Z\ni z,B)\geq \epsilon$, 
  \item the coefficients of $B$ are in $\mathfrak D$, and
  \item $-(K_X+B)$ is nef over $Z$. 
\end{itemize}
Then there exists an effective toric Cartier divisor $H$ on $Z$ such that $z\in \Supp H$ and $(X,B+t f^*H)$ is lc over $z$.
\end{lem}
\begin{proof}
It follows from Lemma \ref{lem:Ambro} by taking $\mathfrak A=(1-\mathfrak D)/\epsilon$.
\end{proof}

In the following theorem, we remove the condition on the coefficients of $B$ in Lemma \ref{lem:bu}, so Conjecture \ref{conj:birkar1} holds in the toric case. The argument of its proof is similar to that of Proposition \ref{prop:4to5} and the case that $B=0$ is known by \cite[Theorem 3.12]{Am22}.

\begin{thm}\label{thm:toric2}
Let $d$ be a natural number, $\epsilon$ be a positive real number. Then there is a positive number $t$ depending only on $d,\epsilon$ satisfying the following. Assume $(X,B)$ is a toric pair, $f:X\to Z$ is a toric contraction with $\dim Z>0$ and $z\in Z$ is a torus invariant closed point such that 
\begin{itemize}
    \item $\mld(X/Z\ni z,B)\geq \epsilon$, and
    \item $-(K_X+B)$ is nef over $Z$. 
\end{itemize}
Then there exists an effective Cartier divisor $H$ on a neighbourhood of $z\in Z$ such that $z\in \Supp H$ and $(X,B+t f^*H)$ is lc over $z$.
\end{thm}
\begin{proof}
Let $t$ be the largest number such that 
$$\mld(X/Z\ni z,B,f^*\mathfrak m_{z}^t)\geq \frac{\epsilon}{2},$$
where $\mathfrak m_{z}$ is the ideal sheaf of the closed point $z\in Z$. By Corollary \ref{cor:idealdiv}, it suffices to show that $t$ is bounded from below by a positive constant depending only on $d,\epsilon$. If $t> 1$, we are done. So we may suppose that $t\leq 1$.

We can find a toric prime divisor $T$ over $X/Z\ni z$ so that
$$a(T,X,B,f^*\mathfrak m_{z}^t)=\frac{\epsilon}{2}.$$
Since $\mld(X/Z\ni z,B)\geq \epsilon$, we have
$$t \cdot \mu_T f^*\mathfrak m_{z}\geq \frac{\epsilon}{2}.$$
So it suffices to show that $\mu_T f^*\mathfrak m_{z}$ is bounded from above by a constant depending only on $d,\epsilon$. 

Let $\pi:Y\to X$ be a toric $\Q$-factorialisation extracting $T$ but no other divisors (if $T$ is not exceptional over $X$, then $\pi$ is a small toric $\Q$-factorialisation). By Lemma \ref{lem:linearsystem} and Remark \ref{rem}, there exists an effective Cartier divisor $D$ on $Z$ such that $z\in \Supp D$,
$$a(T,X,B+tf^*D)=a(T,X,B,f^*\mathfrak m_z^t)=\frac{\epsilon}{2}$$
and 
$$\mld(X/Z\ni z,B+tf^*D)=\mld(X/Z\ni z,B,f^*\mathfrak m_{z}^t)=\frac{\epsilon}{2}.$$
Let $K_Y+\Omega$ be the pullback of $K_X+B+tf^*D$ to $Y$. Then $\Omega$ is effective as $\mu_T \Omega=\frac{\epsilon}{2}\geq 0$. Moreover, $-(K_Y+\Omega)$ is nef over $Z$ and  
$$\mld(Y/Z\ni z, \Omega)= \frac{\epsilon}{2}.$$
As $Y\to Z$ is a toric contraction, $Y$ is of Fano type over $Z$. So $-(K_Y+\Omega)$ is semi-ample over $Z$. It follows that there exists an effective $\R$-divisor $R$ on $Y$ such that $K_Y+\Omega+R\sim_{\R} 0/Z$ and
$$\mld(Y/Z\ni z, \Omega+R)= \frac{\epsilon}{2}.$$

Run an MMP on $-K_Y$ over $Z$ and let $Y'$ be the resulting model with the induced morphism $f':Y'\to Z$. Then $-K_{Y'}$ is nef over $Z$. Moreover, we have
$$\mld (Y'/Z\ni z, 0)\geq \mld(Y'/Z,\Omega'+R')=\frac{\epsilon}{2}, $$
where $\Omega',R'$ are the pushdowns of $\Omega,R$ to $Y'$.
Shrinking $Z$ to an affine toric neighbourhood of $z$, we can apply Lemma \ref{lem:bu} (by taking $X=Y'$ and $B=0$) to deduce that, there are a positive number $u$ depending on $d,\epsilon$ and an effective Cartier divisor $F$ on $Z$ such that $z\in \Supp F$ and $(Y',u\cdot f'^*F)$ is lc over $z$. Then $(Y,u\cdot \pi^*f^*F)$ is lc over $z$, as $Y\dashrightarrow Y'$ is an MMP on $-K_Y$. Thus 
$$\mu_T (\pi^*f^* \mathfrak m_{z})\leq \mu_T (\pi^*f^*F)\leq \frac{1}{u},$$
where $u$ only depends on $d,\epsilon$.
\end{proof}
\begin{proof}[Proof of Theorem \ref{thm:toric}]
It follows from Theorem \ref{thm:toric2} directly.
\end{proof}
\begin{proof}[Proof of Corollary \ref{cor:toric}]
It is a special case of Theorem \ref{thm:toric}.
\end{proof}

\section{Exceptional singularities}
In this section we will treat the case for exceptional singularities.

\begin{thm}\label{thm:exptech}
Let $s$ be a natural number  and $\epsilon,\delta$ be positive real numbers. Then there is a positive real number $t$ depending only on $s,\epsilon,\delta$ satisfying the following. Assume
\begin{itemize}
  \item $(Z \ni z,\Delta)$ is a klt germ of dimension $s$,
  \item $\pi:Y\to Z$ is a plt blow-up of $(Z\ni z,\Delta)$ with the exceptional divisor $E$, 
  \item $a(E,Z,\Delta)\geq \epsilon$,
  \item there is a boundary $\Omega$ on $Z$ such that $(Z\ni z,\Omega)$ is a klt germ and  $\pi:Y\to Z$ is a $\delta$-plt blow-up of $(Z\ni z,\Omega)$.
\end{itemize}
Then there exists an effective Cartier divisor $H$ on a neighbourhood of $z\in Z$ such that $z\in \Supp H$ and $(Z,\Delta+tH)$ is lc near $z$.
\end{thm}
\begin{proof}
Shrinking $Z$ around $z$, we can assume that $(Z,\Omega)$ is klt everywhere. Applying Lemma \ref{lem:projective} (by taking $X=Z$ and $B=\Omega$), we may suppose that $Z$ is projective and $(Z,\Omega)$ is still klt everywhere. By assumption, we have
\begin{itemize}
  \item $-E$ is ample over $Z$, 
  \item $a(E,Z,\Delta)\geq \epsilon$,
  \item $(Y,\Delta_Y+E)$ is plt near $E$, where $\Delta_Y$ is the birational transform of $\Delta$, and
  \item $(Y,\Omega_Y+E)$ is $\delta$-plt, where $\Omega_Y$ is the birational transform of $\Omega$.
\end{itemize}
Then $-(K_Y+\Omega_Y+E)$ is ample over $Z$. 
Let $C=\Omega_Y+(1-a)E$ where $a$ is a sufficiently small positive number. Then  $(Y,C)$ is a klt pair and $-(K_Y+C)$ is ample over $Z$. 

By \cite[Proposition 4.4 and Theorem A.1]{HLS19}, there exists a natural number $n_1$ depending on $s,\epsilon,\delta$ such that $n_1E$ is Cartier. Let $A$ be a very ample divisor on $Z$ such that $\pi^*A-n_1E$, $\pi^*A-(K_Y+C)$ are ample. Then 
$$(2\pi^*A-n_1E)-(K_Y+C)$$
is ample. By the effective base point freeness \cite{Ko93}, there exists a bounded $n_2\in \N$ depending on $s$ such that $2n_2\pi^*A-n_1n_2E$ is base point free. Denote $n=n_1n_2$ and let $G$ be a general element of $|2n_2\pi^*A-nE|$. Then $G+nE\sim 2n_2\pi^*A$ and $(Y,\Delta_Y+E+G)$ is lc over $z$.  Let $H=\pi_*G$, then $H\sim 2n_2A$, which implies that $H$ is an effective Cartier divisor. Moreover, we have $\pi^*H=G+nE$, which implies that $z\in \Supp H$. 
  
Let $K_Y+\widetilde{\Delta}_Y$ be the pullback of $K_Z+\Delta$. Then $\mu_E \widetilde{\Delta}_Y \leq 1-\epsilon$. Thus 
$$\widetilde{\Delta}_Y+\frac{\epsilon}{n} \pi^*H=\widetilde{\Delta}_Y+\epsilon E+\frac{\epsilon}{n} G\leq \Delta_Y+E+G$$
(we may assume that $\epsilon\leq 1$). It follows that $(Y,\widetilde{\Delta}_Y+\frac{\epsilon}{n}\pi^*H)$ is lc over $z$. Take $t=\frac{\epsilon}{n}$, then $(Z,\Delta+tH)$ is lc near $z$ and $t$ only depends on $s,\epsilon,\delta$.
\end{proof}
  
\begin{proof}[Proof of Theorem \ref{thm:exceptional}]
It follows from Theorem \ref{thm:exptech} directly (by taking $\Omega=\Delta$).
\end{proof}

\begin{proof}[Proof of Theorem \ref{thm:exc2}]
When $Z=1$, Conjecture \ref{conj:birkar2} holds obviously. So we may suppose that $\dim Z\geq 2$. By Lemma \ref{lem:has}, $(Z\ni z,\Delta)$ has a plt blow-up $\pi:Y\to Z$. Then $\pi:Y\to Z$ is also a plt blow-up of $(Z\ni z,\Omega)$ since $\Omega\leq \Delta$. By Lemma \ref{lem:exc}, $\pi:Y\to Z$ is a $\delta$-plt blow-up of $(Z\ni z,\Omega)$, where $\delta$ only depends on $s,\mathfrak D$. Thus the theorem follows from Theorem \ref{thm:exptech}.
\end{proof}

\begin{proof}[Proof of Corollary \ref{cor:exc1}]
When $Z\ni z$ is an exceptional singularity, Conjecture \ref{conj:birkar2} follows from Theorem \ref{thm:exc2} (by taking $\Omega=0$) and then Conjecture \ref{conj:birkar1} holds by the proof of Proposition \ref{prop:5to3}.
\end{proof}

\begin{proof}[Proof of Corollary \ref{cor:exc2}]
The corollary follows from the argument in the proofs of Theorem \ref{thm:klt} and \ref{thm:tech}, by replacing the assumption that Conjecture \ref{conj:birkar1} holds when $\dim Z\leq \ell$
with Corollary \ref{cor:dim2} and \ref{cor:exc1}.
\end{proof}

\appendix
\section{Fundamental cycles of quotient surface singularities }
We call $Z\ni z$ a quotient surface singularity if $(Z\ni z)=(\C^2\ni 0)/G$ for some finite subgroup $G\subset GL(2,\C)$. In this appendix, we will show the following proposition.
\begin{prop}\label{prop:quotient}
Let $Z\ni z$ be a quotient surface singularity and let $\pi:Y\to Z$ be the minimal resolution with the exceptional divisor $E$. Denote by $C_f$ the fundamental cycle of $Z\ni z$ on $Y$. Then $C_f\leq 6E$.
\end{prop}
\begin{lem}\label{lem:-2}
  Let $Z\ni z$ and $Z'\ni z'$ be two surface singularities. Let $Y\to Z$ and $Y'\to Z'$ be the minimal resolutions with the exceptional divisors $E=\sum_{i} E_i$ over $Z\ni z$ and $E'=\sum_{i} E'_i$  over $Z'\ni z'$ respectively. Assume that 
  \begin{itemize}
    \item the number of components of $E$ is the same as that of $E'$, and
    \item $E_i\cdot E_j\geq E'_i \cdot E'_j$ for any $i,j$. 
  \end{itemize} 
  Let $C_f=\sum_i c_i E_i$ and $C'_f=\sum_i c'_i E'_i$ be the fundamental cycles on $Y$ and $Y'$ respectively.
  Then $c_i\geq c'_i$ for any $i$.
  \end{lem}
  \begin{proof}
  Denote $\widetilde{C}_f=\sum_i c_i E'_i$. For any component $E'_k$ of $E'$ we have
$$
  \widetilde{C}_f\cdot E'_k=\sum_i c_i E'_i\cdot E'_k\leq \sum_i c_i E_i\cdot E_k=C_f\cdot E_k\leq 0.
$$
  By the definition of the fundamental cycle, $\widetilde{C}_f\geq C'_f$, that is, $c_i\geq c'_i$ for any $i$.
\end{proof}

Let $Z\ni z$ be a quotient surface singularity and let $\pi:Y\to Z$ be the minimal resolution with the exceptional divisor $E$. By Lemma \ref{lem:good}, $E$ is a simple normal crossing divisor on $Y$. The dual graph $\Gamma$ associated to the minimal resolution of  $Z\ni z$ is defined as follows. 
\begin{itemize}
  \item A vertex $v$ of the $\Gamma$ corresponds to  a component $E_v$ of $E$.
  \item Two vertices $v,w$ are connected by $k$ edges if the corresponding components $E_v,E_w$ meet each other in $k$ points.
  \item Each vertex $v$ is decorated with a weight: the self-intersection number of $E_v$.
\end{itemize}
It is clear that the intersection matrix of components of $E$ is determined by the dual graph. Hence the fundamental cycle is determined by the dual graph.

A complete list of quotient surface singularities and the dual graphs associated to their minimal resolutions can be found in \cite[pp.11-16]{BO12}. Using Laufer's algorithm (see Remark \ref{rem:laufer}), we calculate the fundamental cycle $C_f$ on the minimal resolution of each singularity in the list and we conclude that $C_f\leq 6E$.

\medskip

\emph{1. Cyclic quotient singularities $A_{n,q}$, where
$0<q<n$ and $(n,q)=1$}. The dual graph associated to the minimal resolution
of $A_{n,q}$ is
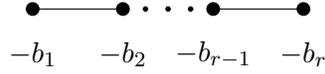
\begin{center}
  \begin{tikzpicture}[scale=1.2] 
    \filldraw[black] (0,0) circle (2pt);
    \filldraw[black] (1,0) circle (2pt);
    \filldraw[black] (2,0) circle (2pt);
    \filldraw[black] (3,0) circle (2pt);
    \filldraw[white] (0,0.5) circle (2pt);
    \draw[thin,-] (0,0)--(1,0);
    \filldraw[black] (1.25,0) circle (0.7pt);
    \filldraw[black] (1.5,0) circle (0.7pt);
    \filldraw[black] (1.75,0) circle (0.7pt);
    \draw[thin,-] (2,0)--(3,0);
    \node at (0,-0.5) {$-b_1$};
    \node at (1,-0.5) {$-b_2$};
    \node at (2,-0.5) {$-b_{r-1}$};
    \node at (3,-0.5) {$-b_r$};
  \end{tikzpicture}
  \end{center}
where the $b_i$ are defined by the Hirzebruch--Jung continued
fraction:
$$ \frac nq = [b_1,b_2,\cdots,b_r] := b_1-
\cfrac{1}{b_2- \cfrac{1}{\ddots- \cfrac{1}{b_r}}}, \quad b_i\geq 2
\text{ for $i=1,\cdots r$}. $$ 
By Lemma \ref{lem:-2}, we only need to consider the case that $b_1=\cdots=b_r=2$. In this case, the fundamental cycle on the minimal resolution is given by

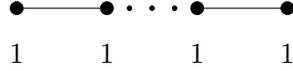
\begin{center}
\begin{tikzpicture}[scale=1.2] 
  \filldraw[black] (0,0) circle (2pt);
  \filldraw[black] (1,0) circle (2pt);
  \filldraw[black] (2,0) circle (2pt);
  \filldraw[black] (3,0) circle (2pt);
  \filldraw[white] (0,0.5) circle (2pt);
  \draw[thin,-] (0,0)--(1,0);
  \filldraw[black] (1.25,0) circle (0.7pt);
  \filldraw[black] (1.5,0) circle (0.7pt);
  \filldraw[black] (1.75,0) circle (0.7pt);
  \draw[thin,-] (2,0)--(3,0);
  \node at (0,-0.5) {1};
  \node at (1,-0.5) {1};
  \node at (2,-0.5) {1};
  \node at (3,-0.5) {1};
\end{tikzpicture}
\end{center}
where the number near each vertex is the coefficient of the corresponding component in the fundamental cycle.

\medskip

\emph{2. Dihedral singularities $D_{n,q}$, where $1<q<n$
and $(n,q)=1$}. The dual graph associated to the minimal resolution of $D_{n,q}$ is
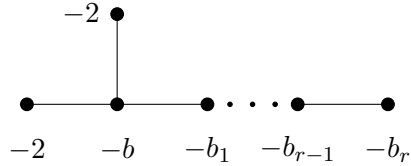
\begin{center}
  \begin{tikzpicture}[scale=1.2] 
    \filldraw[black] (0,0) circle (2pt);
    \filldraw[black] (1,0) circle (2pt);
    \filldraw[black] (2,0) circle (2pt);
    \filldraw[black] (3,0) circle (2pt);
    \filldraw[black] (4,0) circle (2pt);
    \filldraw[black] (1,1) circle (2pt);
    \filldraw[white] (1,1.5) circle (2pt);
    \draw[thin,-] (0,0)--(1,0)--(2,0);
    \filldraw[black] (2.25,0) circle (0.7pt);
    \filldraw[black] (2.5,0) circle (0.7pt);
    \filldraw[black] (2.75,0) circle (0.7pt);
    \draw[thin,-] (3,0)--(4,0);
    \draw[thin,-] (1,0)--(1,1);
    \node at (0,-0.5) {$-2$};
    \node at (1,-0.5) {$-b$};
    \node at (2,-0.5) {$-b_1$};
    \node at (3,-0.5) {$-b_{r-1}$};
    \node at (4,-0.5) {$-b_r$};
    \node at (0.6,1) {$-2$};
  \end{tikzpicture}
  \end{center}
where $b,b_i,i=1,\cdots,r$ are defined by
$$ \frac nq = [b,b_1,\cdots,b_r],\quad b\geq2,\ b_i\geq 2 \text{ for $i=1,\cdots r$}. $$
By Lemma \ref{lem:-2}, we only need to consider the case that $b=b_1=\cdots=b_r=2$. In this case, the fundamental cycle on the minimal resolution is given by

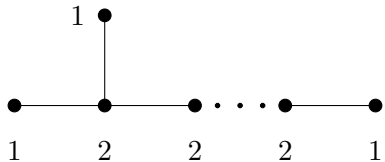
\begin{center}
  \begin{tikzpicture}[scale=1.2] 
    \filldraw[black] (0,0) circle (2pt);
    \filldraw[black] (1,0) circle (2pt);
    \filldraw[black] (2,0) circle (2pt);
    \filldraw[black] (3,0) circle (2pt);
    \filldraw[black] (4,0) circle (2pt);
    \filldraw[black] (1,1) circle (2pt);
    \filldraw[white] (1,1.5) circle (2pt);
    \draw[thin,-] (0,0)--(1,0)--(2,0);
    \filldraw[black] (2.25,0) circle (0.7pt);
    \filldraw[black] (2.5,0) circle (0.7pt);
    \filldraw[black] (2.75,0) circle (0.7pt);
    \draw[thin,-] (3,0)--(4,0);
    \draw[thin,-] (1,0)--(1,1);
    \node at (0,-0.5) {1};
    \node at (1,-0.5) {2};
    \node at (2,-0.5) {2};
    \node at (3,-0.5) {2};
    \node at (4,-0.5) {1};
    \node at (0.7,1) {1};
  \end{tikzpicture}
  \end{center}

\medskip

\emph{3. Tetrahedral singularities $T_m$, where $m=1,3,5
\mod 6$}. The dual graphs associated to their minimal resolutions are given in Table~\ref{tab:tet}. By Lemma \ref{lem:-2}, we only need to consider the case that $b=2$. The fundamental cycles for $b=2$  are also given in Table~\ref{tab:tet}.

\medskip

\emph{\it 4. Octahedral singularities $O_m$, where
$(m,6)=1$}. The dual graphs associated to their minimal resolutions
are given in Table~\ref{tab:oct}. By Lemma \ref{lem:-2}, we only need to consider the case that $b=2$. The fundamental cycles for $b=2$  are also given in Table~\ref{tab:oct}.

\medskip

\emph{5. Icosahedral singularities $I_m$, where
$(m,30)=1$}. The dual graphs associated to their minimal resolutions are given in Table \ref{tab:icos}. By Lemma \ref{lem:-2}, we only need to consider the case that $b=2$. The fundamental cycles for $b=2$ are also given in Tables \ref{tab:icos}.

\begin{table}[htbp]
  \caption{Dual graphs and fundamental cycles for
  tetrahedral singularities.} \label{tab:tet}
  \begin{tabular}{|c|c|c|}
  \hline
  $b\geq 2$ & Dual graph & Fundamental cycle for $b=2$ \\
  \hline $m=6(b-2)+1$ &
  {\begin{tabular}{c}

    \begin{tikzpicture}[scale=1] 
      \filldraw[black] (0,0) circle (2pt);
      \filldraw[black] (1,0) circle (2pt);
      \filldraw[black] (2,0) circle (2pt);
      \filldraw[black] (3,0) circle (2pt);
      \filldraw[black] (4,0) circle (2pt);
      \filldraw[black] (2,1) circle (2pt);
      \draw[thin,-] (0,0)--(1,0)--(2,0)--(3,0);
      \draw[thin,-] (3,0)--(4,0);
      \draw[thin,-] (2,0)--(2,1);
      \node at (0,-0.5) {\scriptsize  $-2$};
      \node at (1,-0.5) {\scriptsize  $-2$};
      \node at (2,-0.5) {\scriptsize  $-b$};
      \node at (3,-0.5) {\scriptsize  $-2$};
      \node at (4,-0.5) {\scriptsize  $-2$};
      \node at (1.6,1) {\scriptsize $-2$};
      \node at (0.7,1.3) {};
    \end{tikzpicture}

  \end{tabular}}
   &
  {\begin{tabular}{c}
    
    \begin{tikzpicture}[scale=1] 
      \filldraw[black] (0,0) circle (2pt);
      \filldraw[black] (1,0) circle (2pt);
      \filldraw[black] (2,0) circle (2pt);
      \filldraw[black] (3,0) circle (2pt);
      \filldraw[black] (4,0) circle (2pt);
      \filldraw[black] (2,1) circle (2pt);
      \draw[thin,-] (0,0)--(1,0)--(2,0)--(3,0);
      \draw[thin,-] (3,0)--(4,0);
      \draw[thin,-] (2,0)--(2,1);
      \node at (0,-0.5) {\scriptsize  1};
      \node at (1,-0.5) {\scriptsize  2};
      \node at (2,-0.5) {\scriptsize  3};
      \node at (3,-0.5) {\scriptsize  2};
      \node at (4,-0.5) {\scriptsize  1};
      \node at (1.7,1) {\scriptsize 2};
      \node at (0.7,1.3) {};
    \end{tikzpicture}

  \end{tabular}} \\
  \hline $m=6(b-2)+3$ &
  {\begin{tabular}{c}

    \begin{tikzpicture}[scale=1] 
      \filldraw[black] (0,0) circle (2pt);
      \filldraw[black] (1,0) circle (2pt);
      \filldraw[black] (2,0) circle (2pt);
      \filldraw[black] (3,0) circle (2pt);
      \filldraw[black] (2,1) circle (2pt);
      \draw[thin,-] (0,0)--(1,0)--(2,0)--(3,0);
      \draw[thin,-] (2,0)--(2,1);
      \node at (0,-0.5) {\scriptsize  $-2$};
      \node at (1,-0.5) {\scriptsize  $-2$};
      \node at (2,-0.5) {\scriptsize  $-b$};
      \node at (3,-0.5) {\scriptsize  $-3$};
      \node at (1.6,1) {\scriptsize $-2$};
      \node at (0.7,1.3) {};
    \end{tikzpicture}

  \end{tabular}} &
  {\begin{tabular}{c}
    
    \begin{tikzpicture}[scale=1] 
      \filldraw[black] (0,0) circle (2pt);
      \filldraw[black] (1,0) circle (2pt);
      \filldraw[black] (2,0) circle (2pt);
      \filldraw[black] (3,0) circle (2pt);
      \filldraw[black] (2,1) circle (2pt);
      \draw[thin,-] (0,0)--(1,0)--(2,0)--(3,0);
      \draw[thin,-] (2,0)--(2,1);
      \node at (0,-0.5) {\scriptsize  1};
      \node at (1,-0.5) {\scriptsize  2};
      \node at (2,-0.5) {\scriptsize  2};
      \node at (3,-0.5) {\scriptsize  1};
      \node at (1.7,1) {\scriptsize 1};
      \node at (0.7,1.3) {};
    \end{tikzpicture}

  \end{tabular}}
  \\
  \hline $m=6(b-2)+5$ &
  {\begin{tabular}{c}
    \begin{tikzpicture}[scale=1] 
      \filldraw[black] (0,0) circle (2pt);
      \filldraw[black] (1,0) circle (2pt);
      \filldraw[black] (2,0) circle (2pt);
      \filldraw[black] (1,1) circle (2pt);
      \draw[thin,-] (0,0)--(1,0)--(2,0);
      \draw[thin,-] (1,0)--(1,1);
      \node at (0,-0.5) {\scriptsize  $-3$};
      \node at (1,-0.5) {\scriptsize  $-b$};
      \node at (2,-0.5) {\scriptsize  $-3$};
      \node at (0.6,1) {\scriptsize $-2$};
      \node at (0.7,1.3) {};
    \end{tikzpicture}
  \end{tabular}}
  &
  {\begin{tabular}{c}
    
    \begin{tikzpicture}[scale=1] 
      \filldraw[black] (0,0) circle (2pt);
      \filldraw[black] (1,0) circle (2pt);
      \filldraw[black] (2,0) circle (2pt);
      \filldraw[black] (1,1) circle (2pt);
      \draw[thin,-] (0,0)--(1,0)--(2,0);
      \draw[thin,-] (1,0)--(1,1);
      \node at (0,-0.5) {\scriptsize  1};
      \node at (1,-0.5) {\scriptsize  2};
      \node at (2,-0.5) {\scriptsize  1};
      \node at (0.7,1) {\scriptsize 1};
      \node at (0.7,1.3) {};
    \end{tikzpicture}

  \end{tabular}}
  \\
  \hline
  \end{tabular}
  \end{table}

  \begin{table}[htbp]
    \caption{Dual graphs and fundamental cycles for
    octahedral singularities.} \label{tab:oct}
    \begin{tabular}{|c|c|c|}
    \hline
    $b\geq 2$ & Dual graph & Fundamental cycle for $b=2$ \\
    \hline $m=12(b-2)+1$ &
    {\begin{tabular}{c}

      \begin{tikzpicture}[scale=0.8] 
        \filldraw[black] (0,0) circle (2pt);
        \filldraw[black] (1,0) circle (2pt);
        \filldraw[black] (2,0) circle (2pt);
        \filldraw[black] (3,0) circle (2pt);
        \filldraw[black] (4,0) circle (2pt);
        \filldraw[black] (5,0) circle (2pt);
        \filldraw[black] (2,1) circle (2pt);
        \draw[thin,-] (0,0)--(1,0)--(2,0)--(3,0);
        \draw[thin,-] (3,0)--(4,0)--(5,0);
        \draw[thin,-] (2,0)--(2,1);
        \node at (0,-0.5) {\scriptsize  $-2$};
        \node at (1,-0.5) {\scriptsize  $-2$};
        \node at (2,-0.5) {\scriptsize  $-b$};
        \node at (3,-0.5) {\scriptsize  $-2$};
        \node at (4,-0.5) {\scriptsize  $-2$};
        \node at (5,-0.5) {\scriptsize  $-2$};
        \node at (1.6,1) {\scriptsize $-2$};
        \node at (0.7,1.3) {};
      \end{tikzpicture}

    \end{tabular}} &
    {\begin{tabular}{c}
    
      \begin{tikzpicture}[scale=0.8] 
        \filldraw[black] (0,0) circle (2pt);
        \filldraw[black] (1,0) circle (2pt);
        \filldraw[black] (2,0) circle (2pt);
        \filldraw[black] (3,0) circle (2pt);
        \filldraw[black] (4,0) circle (2pt);
        \filldraw[black] (5,0) circle (2pt);
        \filldraw[black] (2,1) circle (2pt);
        \draw[thin,-] (0,0)--(1,0)--(2,0)--(3,0);
        \draw[thin,-] (3,0)--(4,0)--(5,0);
        \draw[thin,-] (2,0)--(2,1);
        \node at (0,-0.5) {\scriptsize  2};
        \node at (1,-0.5) {\scriptsize  3};
        \node at (2,-0.5) {\scriptsize  4};
        \node at (3,-0.5) {\scriptsize  3};
        \node at (4,-0.5) {\scriptsize  2};
        \node at (5,-0.5) {\scriptsize  1};
        \node at (1.7,1) {\scriptsize 2};
        \node at (0.7,1.3) {};
      \end{tikzpicture}

    \end{tabular}}\\
    \hline $m=12(b-2)+5$ &
    {\begin{tabular}{c}

      \begin{tikzpicture}[scale=1] 
        \filldraw[black] (0,0) circle (2pt);
        \filldraw[black] (1,0) circle (2pt);
        \filldraw[black] (2,0) circle (2pt);
        \filldraw[black] (3,0) circle (2pt);
        \filldraw[black] (4,0) circle (2pt);
        \filldraw[black] (1,1) circle (2pt);
        \draw[thin,-] (0,0)--(1,0)--(2,0)--(3,0)--(4,0);
        \draw[thin,-] (1,0)--(1,1);
        \node at (0,-0.5) {\scriptsize  $-3$};
        \node at (1,-0.5) {\scriptsize  $-b$};
        \node at (2,-0.5) {\scriptsize  $-2$};
        \node at (3,-0.5) {\scriptsize  $-2$};
        \node at (4,-0.5) {\scriptsize  $-2$};
        \node at (0.6,1) {\scriptsize $-2$};
        \node at (0.7,1.3) {};
      \end{tikzpicture}
    
    \end{tabular}} &
    {\begin{tabular}{c}
    
      \begin{tikzpicture}[scale=1] 
        \filldraw[black] (0,0) circle (2pt);
        \filldraw[black] (1,0) circle (2pt);
        \filldraw[black] (2,0) circle (2pt);
        \filldraw[black] (3,0) circle (2pt);
        \filldraw[black] (4,0) circle (2pt);
        \filldraw[black] (1,1) circle (2pt);
        \draw[thin,-] (0,0)--(1,0)--(2,0)--(3,0)--(4,0);
        \draw[thin,-] (1,0)--(1,1);
        \node at (0,-0.5) {\scriptsize  1};
        \node at (1,-0.5) {\scriptsize  2};
        \node at (2,-0.5) {\scriptsize  2};
        \node at (3,-0.5) {\scriptsize  2};
        \node at (4,-0.5) {\scriptsize  1};
        \node at (0.7,1) {\scriptsize 1};
        \node at (0.7,1.3) {};
      \end{tikzpicture}
    
    \end{tabular}} \\
    \hline $m=12(b-2)+7$ &
    {\begin{tabular}{c}

      \begin{tikzpicture}[scale=1] 
        \filldraw[black] (0,0) circle (2pt);
        \filldraw[black] (1,0) circle (2pt);
        \filldraw[black] (2,0) circle (2pt);
        \filldraw[black] (3,0) circle (2pt);
        \filldraw[black] (2,1) circle (2pt);
        \draw[thin,-] (0,0)--(1,0)--(2,0)--(3,0);
        \draw[thin,-] (2,0)--(2,1);
        \node at (0,-0.5) {\scriptsize  $-2$};
        \node at (1,-0.5) {\scriptsize  $-2$};
        \node at (2,-0.5) {\scriptsize  $-b$};
        \node at (3,-0.5) {\scriptsize  $-4$};
        \node at (1.6,1) {\scriptsize $-2$};
        \node at (0.7,1.3) {};
      \end{tikzpicture}
    
    \end{tabular}}
     &
    {\begin{tabular}{c}
      
      \begin{tikzpicture}[scale=1] 
        \filldraw[black] (0,0) circle (2pt);
        \filldraw[black] (1,0) circle (2pt);
        \filldraw[black] (2,0) circle (2pt);
        \filldraw[black] (3,0) circle (2pt);
        \filldraw[black] (2,1) circle (2pt);
        \draw[thin,-] (0,0)--(1,0)--(2,0)--(3,0);
        \draw[thin,-] (2,0)--(2,1);
        \node at (0,-0.5) {\scriptsize  1};
        \node at (1,-0.5) {\scriptsize  2};
        \node at (2,-0.5) {\scriptsize  2};
        \node at (3,-0.5) {\scriptsize  1};
        \node at (1.7,1) {\scriptsize 1};
        \node at (0.7,1.3) {};
      \end{tikzpicture}

    \end{tabular}} \\
    \hline $m=12(b-2)+11$ &
    {\begin{tabular}{c}

      \begin{tikzpicture}[scale=1] 
        \filldraw[black] (0,0) circle (2pt);
        \filldraw[black] (1,0) circle (2pt);
        \filldraw[black] (2,0) circle (2pt);
        \filldraw[black] (1,1) circle (2pt);
        \draw[thin,-] (0,0)--(1,0)--(2,0);
        \draw[thin,-] (1,0)--(1,1);
        \node at (0,-0.5) {\scriptsize  $-3$};
        \node at (1,-0.5) {\scriptsize  $-b$};
        \node at (2,-0.5) {\scriptsize  $-4$};
        \node at (0.6,1) {\scriptsize $-2$};
        \node at (0.7,1.3) {};
      \end{tikzpicture}

    \end{tabular}}
     &
    {\begin{tabular}{c}
    
      \begin{tikzpicture}[scale=1] 
        \filldraw[black] (0,0) circle (2pt);
        \filldraw[black] (1,0) circle (2pt);
        \filldraw[black] (2,0) circle (2pt);
        \filldraw[black] (1,1) circle (2pt);
        \draw[thin,-] (0,0)--(1,0)--(2,0);
        \draw[thin,-] (1,0)--(1,1);
        \node at (0,-0.5) {\scriptsize  1};
        \node at (1,-0.5) {\scriptsize  2};
        \node at (2,-0.5) {\scriptsize  1};
        \node at (0.7,1) {\scriptsize 1};
        \node at (0.7,1.3) {};
      \end{tikzpicture}
    
    \end{tabular}} \\
    \hline
    \end{tabular}
    \end{table}

\begin{table}[htbp]
\caption{Dual graphs and fundamental cycles for
icosahedral singularities.} \label{tab:icos}
\begin{tabular}{|c|c|c|}
\hline
$b \geq 2$ & Dual graph & Fundamental cycle for $b=2$ \\
\hline $m=30(b-2)+1$ &
{\begin{tabular}{c}
  \begin{tikzpicture}[scale=0.6] 
    \filldraw[black] (0,0) circle (2pt);
    \filldraw[black] (1,0) circle (2pt);
    \filldraw[black] (2,0) circle (2pt);
    \filldraw[black] (3,0) circle (2pt);
    \filldraw[black] (4,0) circle (2pt);
    \filldraw[black] (5,0) circle (2pt);
    \filldraw[black] (6,0) circle (2pt);
    \filldraw[black] (2,1) circle (2pt);
    \draw[thin,-] (0,0)--(1,0)--(2,0)--(3,0);
    \draw[thin,-] (3,0)--(4,0)--(5,0)--(6,0);
    \draw[thin,-] (2,0)--(2,1);
    \node at (0,-0.5) {\scriptsize  $-2$};
    \node at (1,-0.5) {\scriptsize  $-2$};
    \node at (2,-0.5) {\scriptsize  $-b$};
    \node at (3,-0.5) {\scriptsize  $-2$};
    \node at (4,-0.5) {\scriptsize  $-2$};
    \node at (5,-0.5) {\scriptsize  $-2$};
    \node at (6,-0.5) {\scriptsize  $-2$};
    \node at (1.5,1) {\scriptsize $-2$};
    \node at (0.7,1.3) {};
  \end{tikzpicture}

\end{tabular}}
 &
{\begin{tabular}{c}

  \begin{tikzpicture}[scale=0.6] 
    \filldraw[black] (0,0) circle (2pt);
    \filldraw[black] (1,0) circle (2pt);
    \filldraw[black] (2,0) circle (2pt);
    \filldraw[black] (3,0) circle (2pt);
    \filldraw[black] (4,0) circle (2pt);
    \filldraw[black] (5,0) circle (2pt);
    \filldraw[black] (6,0) circle (2pt);
    \filldraw[black] (2,1) circle (2pt);
    \draw[thin,-] (0,0)--(1,0)--(2,0)--(3,0);
    \draw[thin,-] (3,0)--(4,0)--(5,0)--(6,0);
    \draw[thin,-] (2,0)--(2,1);
    \node at (0,-0.5) {\scriptsize  2};
    \node at (1,-0.5) {\scriptsize  4};
    \node at (2,-0.5) {\scriptsize  6};
    \node at (3,-0.5) {\scriptsize  5};
    \node at (4,-0.5) {\scriptsize  4};
    \node at (5,-0.5) {\scriptsize  3};
    \node at (6,-0.5) {\scriptsize  2};
    \node at (1.7,1) {\scriptsize 3};
    \node at (0.7,1.3) {};
  \end{tikzpicture}

\end{tabular}}
 \\
\hline $m=30(b-2)+7$ &
{\begin{tabular}{c}

  \begin{tikzpicture}[scale=1] 
    \filldraw[black] (0,0) circle (2pt);
    \filldraw[black] (1,0) circle (2pt);
    \filldraw[black] (2,0) circle (2pt);
    \filldraw[black] (3,0) circle (2pt);
    \filldraw[black] (4,0) circle (2pt);
    \filldraw[black] (2,1) circle (2pt);
    \draw[thin,-] (0,0)--(1,0)--(2,0)--(3,0);
    \draw[thin,-] (3,0)--(4,0);
    \draw[thin,-] (2,0)--(2,1);
    \node at (0,-0.5) {\scriptsize  $-2$};
    \node at (1,-0.5) {\scriptsize  $-2$};
    \node at (2,-0.5) {\scriptsize  $-b$};
    \node at (3,-0.5) {\scriptsize  $-2$};
    \node at (4,-0.5) {\scriptsize  $-3$};
    \node at (1.6,1) {\scriptsize $-2$};
    \node at (0.7,1.3) {};
  \end{tikzpicture}

\end{tabular}}
 &
{\begin{tabular}{c}

  \begin{tikzpicture}[scale=1] 
    \filldraw[black] (0,0) circle (2pt);
    \filldraw[black] (1,0) circle (2pt);
    \filldraw[black] (2,0) circle (2pt);
    \filldraw[black] (3,0) circle (2pt);
    \filldraw[black] (4,0) circle (2pt);
    \filldraw[black] (2,1) circle (2pt);
    \draw[thin,-] (0,0)--(1,0)--(2,0)--(3,0);
    \draw[thin,-] (3,0)--(4,0);
    \draw[thin,-] (2,0)--(2,1);
    \node at (0,-0.5) {\scriptsize  1};
    \node at (1,-0.5) {\scriptsize  2};
    \node at (2,-0.5) {\scriptsize  3};
    \node at (3,-0.5) {\scriptsize  2};
    \node at (4,-0.5) {\scriptsize  1};
    \node at (1.7,1) {\scriptsize 2};
    \node at (0.7,1.3) {};
  \end{tikzpicture}

\end{tabular}}
 \\
\hline $m=30(b-2)+11$ &
{\begin{tabular}{c}
  \begin{tikzpicture}[scale=0.8] 
    \filldraw[black] (0,0) circle (2pt);
    \filldraw[black] (1,0) circle (2pt);
    \filldraw[black] (2,0) circle (2pt);
    \filldraw[black] (3,0) circle (2pt);
    \filldraw[black] (4,0) circle (2pt);
    \filldraw[black] (5,0) circle (2pt);
    \filldraw[black] (1,1) circle (2pt);
    \draw[thin,-] (0,0)--(1,0)--(2,0)--(3,0)--(4,0)--(5,0);
    \draw[thin,-] (1,0)--(1,1);
    \node at (0,-0.5) {\scriptsize  $-3$};
    \node at (1,-0.5) {\scriptsize  $-b$};
    \node at (2,-0.5) {\scriptsize  $-2$};
    \node at (3,-0.5) {\scriptsize  $-2$};
    \node at (4,-0.5) {\scriptsize  $-2$};
    \node at (5,-0.5) {\scriptsize  $-2$};
    \node at (0.6,1) {\scriptsize $-2$};
    \node at (0.7,1.3) {};
  \end{tikzpicture}

\end{tabular}}
 &
{\begin{tabular}{c}

  \begin{tikzpicture}[scale=0.8] 
    \filldraw[black] (0,0) circle (2pt);
    \filldraw[black] (1,0) circle (2pt);
    \filldraw[black] (2,0) circle (2pt);
    \filldraw[black] (3,0) circle (2pt);
    \filldraw[black] (4,0) circle (2pt);
    \filldraw[black] (5,0) circle (2pt);
    \filldraw[black] (1,1) circle (2pt);
    \draw[thin,-] (0,0)--(1,0)--(2,0)--(3,0)--(4,0)--(5,0);
    \draw[thin,-] (1,0)--(1,1);
    \node at (0,-0.5) {\scriptsize  1};
    \node at (1,-0.5) {\scriptsize  2};
    \node at (2,-0.5) {\scriptsize  2};
    \node at (3,-0.5) {\scriptsize  2};
    \node at (4,-0.5) {\scriptsize  2};
    \node at (5,-0.5) {\scriptsize  1};
    \node at (0.7,1) {\scriptsize 1};
    \node at (0.7,1.3) {};
  \end{tikzpicture}

\end{tabular}}
 \\
\hline $m=30(b-2)+13$ &
{\begin{tabular}{c}

  \begin{tikzpicture}[scale=1] 
    \filldraw[black] (0,0) circle (2pt);
    \filldraw[black] (1,0) circle (2pt);
    \filldraw[black] (2,0) circle (2pt);
    \filldraw[black] (3,0) circle (2pt);
    \filldraw[black] (4,0) circle (2pt);
    \filldraw[black] (2,1) circle (2pt);
    \draw[thin,-] (0,0)--(1,0)--(2,0)--(3,0);
    \draw[thin,-] (3,0)--(4,0);
    \draw[thin,-] (2,0)--(2,1);
    \node at (0,-0.5) {\scriptsize  $-2$};
    \node at (1,-0.5) {\scriptsize  $-2$};
    \node at (2,-0.5) {\scriptsize  $-b$};
    \node at (3,-0.5) {\scriptsize  $-3$};
    \node at (4,-0.5) {\scriptsize  $-2$};
    \node at (1.6,1) {\scriptsize $-2$};
    \node at (0.7,1.3) {};
  \end{tikzpicture}

\end{tabular}}
 &
{\begin{tabular}{c}

  \begin{tikzpicture}[scale=1] 
    \filldraw[black] (0,0) circle (2pt);
    \filldraw[black] (1,0) circle (2pt);
    \filldraw[black] (2,0) circle (2pt);
    \filldraw[black] (3,0) circle (2pt);
    \filldraw[black] (4,0) circle (2pt);
    \filldraw[black] (2,1) circle (2pt);
    \draw[thin,-] (0,0)--(1,0)--(2,0)--(3,0);
    \draw[thin,-] (3,0)--(4,0);
    \draw[thin,-] (2,0)--(2,1);
    \node at (0,-0.5) {\scriptsize  1};
    \node at (1,-0.5) {\scriptsize  2};
    \node at (2,-0.5) {\scriptsize  2};
    \node at (3,-0.5) {\scriptsize  1};
    \node at (4,-0.5) {\scriptsize  1};
    \node at (1.7,1) {\scriptsize 1};
    \node at (0.7,1.3) {};
  \end{tikzpicture}

\end{tabular}}
 \\
\hline $m=30(b-2)+17$ &
{\begin{tabular}{c}

  \begin{tikzpicture}[scale=1] 
    \filldraw[black] (0,0) circle (2pt);
    \filldraw[black] (1,0) circle (2pt);
    \filldraw[black] (2,0) circle (2pt);
    \filldraw[black] (3,0) circle (2pt);
    \filldraw[black] (1,1) circle (2pt);
    \draw[thin,-] (0,0)--(1,0)--(2,0)--(3,0);
    \draw[thin,-] (1,0)--(1,1);
    \node at (0,-0.5) {\scriptsize  $-3$};
    \node at (1,-0.5) {\scriptsize  $-b$};
    \node at (2,-0.5) {\scriptsize  $-2$};
    \node at (3,-0.5) {\scriptsize  $-3$};
    \node at (0.6,1) {\scriptsize $-2$};
    \node at (0.7,1.3) {};
  \end{tikzpicture}

\end{tabular}}
&
{\begin{tabular}{c}

  \begin{tikzpicture}[scale=1] 
    \filldraw[black] (0,0) circle (2pt);
    \filldraw[black] (1,0) circle (2pt);
    \filldraw[black] (2,0) circle (2pt);
    \filldraw[black] (3,0) circle (2pt);
    \filldraw[black] (1,1) circle (2pt);
    \draw[thin,-] (0,0)--(1,0)--(2,0)--(3,0);
    \draw[thin,-] (1,0)--(1,1);
    \node at (0,-0.5) {\scriptsize  1};
    \node at (1,-0.5) {\scriptsize  2};
    \node at (2,-0.5) {\scriptsize  2};
    \node at (3,-0.5) {\scriptsize  1};
    \node at (0.7,1) {\scriptsize 1};
    \node at (0.7,1.3) {};
  \end{tikzpicture}

\end{tabular}} \\
\hline $m=30(b-2)+19$ &
{\begin{tabular}{c}

  \begin{tikzpicture}[scale=1] 
    \filldraw[black] (0,0) circle (2pt);
    \filldraw[black] (1,0) circle (2pt);
    \filldraw[black] (2,0) circle (2pt);
    \filldraw[black] (3,0) circle (2pt);
    \filldraw[black] (2,1) circle (2pt);
    \draw[thin,-] (0,0)--(1,0)--(2,0)--(3,0);
    \draw[thin,-] (2,0)--(2,1);
    \node at (0,-0.5) {\scriptsize  $-2$};
    \node at (1,-0.5) {\scriptsize  $-2$};
    \node at (2,-0.5) {\scriptsize  $-b$};
    \node at (3,-0.5) {\scriptsize  $-5$};
    \node at (1.6,1) {\scriptsize $-2$};
    \node at (0.7,1.3) {};
  \end{tikzpicture}

\end{tabular}}
 &
{\begin{tabular}{c}

  \begin{tikzpicture}[scale=1] 
    \filldraw[black] (0,0) circle (2pt);
    \filldraw[black] (1,0) circle (2pt);
    \filldraw[black] (2,0) circle (2pt);
    \filldraw[black] (3,0) circle (2pt);
    \filldraw[black] (2,1) circle (2pt);
    \draw[thin,-] (0,0)--(1,0)--(2,0)--(3,0);
    \draw[thin,-] (2,0)--(2,1);
    \node at (0,-0.5) {\scriptsize  1};
    \node at (1,-0.5) {\scriptsize  2};
    \node at (2,-0.5) {\scriptsize  2};
    \node at (3,-0.5) {\scriptsize  1};
    \node at (1.7,1) {\scriptsize 1};
    \node at (0.7,1.3) {};
  \end{tikzpicture}

\end{tabular}} \\
\hline $m=30(b-2)+23$ &
{\begin{tabular}{c}

  \begin{tikzpicture}[scale=1] 
    \filldraw[black] (0,0) circle (2pt);
    \filldraw[black] (1,0) circle (2pt);
    \filldraw[black] (2,0) circle (2pt);
    \filldraw[black] (3,0) circle (2pt);
    \filldraw[black] (1,1) circle (2pt);
    \draw[thin,-] (0,0)--(1,0)--(2,0)--(3,0);
    \draw[thin,-] (1,0)--(1,1);
    \node at (0,-0.5) {\scriptsize  $-3$};
    \node at (1,-0.5) {\scriptsize  $-b$};
    \node at (2,-0.5) {\scriptsize  $-3$};
    \node at (3,-0.5) {\scriptsize  $-2$};
    \node at (0.6,1) {\scriptsize $-2$};
    \node at (0.7,1.3) {};
  \end{tikzpicture}

\end{tabular}}
 &
{\begin{tabular}{c}

  \begin{tikzpicture}[scale=1] 
    \filldraw[black] (0,0) circle (2pt);
    \filldraw[black] (1,0) circle (2pt);
    \filldraw[black] (2,0) circle (2pt);
    \filldraw[black] (3,0) circle (2pt);
    \filldraw[black] (1,1) circle (2pt);
    \draw[thin,-] (0,0)--(1,0)--(2,0)--(3,0);
    \draw[thin,-] (1,0)--(1,1);
    \node at (0,-0.5) {\scriptsize  1};
    \node at (1,-0.5) {\scriptsize  2};
    \node at (2,-0.5) {\scriptsize  1};
    \node at (3,-0.5) {\scriptsize  1};
    \node at (0.7,1) {\scriptsize 1};
    \node at (0.7,1.3) {};
  \end{tikzpicture}

\end{tabular}}
 \\
\hline $m=30(b-2)+29$ &
{\begin{tabular}{c}

  \begin{tikzpicture}[scale=1] 
    \filldraw[black] (0,0) circle (2pt);
    \filldraw[black] (1,0) circle (2pt);
    \filldraw[black] (2,0) circle (2pt);
    \filldraw[black] (1,1) circle (2pt);
    \draw[thin,-] (0,0)--(1,0)--(2,0);
    \draw[thin,-] (1,0)--(1,1);
    \node at (0,-0.5) {\scriptsize  $-3$};
    \node at (1,-0.5) {\scriptsize  $-b$};
    \node at (2,-0.5) {\scriptsize  $-5$};
    \node at (0.6,1) {\scriptsize $-2$};
    \node at (0.7,1.3) {};
  \end{tikzpicture}

\end{tabular}}
&
{\begin{tabular}{c}

  \begin{tikzpicture}[scale=1] 
    \filldraw[black] (0,0) circle (2pt);
    \filldraw[black] (1,0) circle (2pt);
    \filldraw[black] (2,0) circle (2pt);
    \filldraw[black] (1,1) circle (2pt);
    \draw[thin,-] (0,0)--(1,0)--(2,0);
    \draw[thin,-] (1,0)--(1,1);
    \node at (0,-0.5) {\scriptsize  1};
    \node at (1,-0.5) {\scriptsize  2};
    \node at (2,-0.5) {\scriptsize  1};
    \node at (0.7,1) {\scriptsize 1};
    \node at (0.7,1.3) {};
  \end{tikzpicture}

\end{tabular}}
 \\
\hline
\end{tabular}
\end{table}


\end{document}